\documentclass[preprint, noshowframe]{imsart}

\RequirePackage[numbers]{natbib}
\RequirePackage[colorlinks,citecolor=blue,urlcolor=blue]{hyperref}
\RequirePackage{graphicx}

\usepackage[utf8]{inputenc}  

\usepackage{nicefrac}
\usepackage{amsthm}
\usepackage{amsmath,amssymb}
\usepackage{dsfont}
\usepackage{array}
\usepackage{booktabs}
\usepackage{graphicx}
\usepackage{bbm}   
\usepackage{comment}
\usepackage{url}
\usepackage{nameref}
\usepackage{ifthen}
\usepackage{natbib}
\usepackage{color}
\usepackage{epstopdf}
\usepackage{caption}
\usepackage{subcaption}
\usepackage{arydshln}
\usepackage{longtable}
\usepackage{marvosym}
\usepackage{cases}
\usepackage{enumitem}
\usepackage{easy-todo}

\startlocaldefs

\newtheorem{theorem}{Theorem}
\newtheorem{lemma}{Lemma}

\newtheorem{corollary}[theorem]{Corollary}
\newtheorem{assumption}{Assumption}
\newtheorem{remark}[lemma]{Remark}

\newtheorem{example}[lemma]{Example}

\newcommand{\mc}{\mathcal}
\newcommand{\dx}{\mathrm{d}}
\newcommand{\dd}{\mathrm{d}}
\newcommand{\e}{\varepsilon}
\newcommand{\bs}[1]{\boldsymbol{#1}} 
\newcommand{\Quad}{\qquad \quad} 
\newcommand{\QQuad}{\qquad \qquad}

\newcommand{\Tinv}{T^{\mathrm{inv}}}

\newcommand{\expec}{{\mathbb{E}}}
\newcommand{\prob}{{\mathbb{P}}}

\newcommand{\ind}{\text{\textbf{1}}} 

\newcommand{\one}{\mathbbm{1}}

\newcommand{\floor}[1]{\left\lfloor#1\right\rfloor}

\newcommand{\abs}[1]{\lvert #1 \rvert}
\newcommand{\KL}{\operatorname{KL}}

\newcommand{\card}{\operatorname{card}}

\newcommand{\Oop}[1]{{\operatorname{\mathcal{O}}\left(#1\right)}}

\newcommand{\cO}{\ensuremath{\mathcal{O}}} 

\newcommand{\norm}[1]{{\lVert#1\rVert}} 
\newcommand{\normb}[1]{{\big\lVert#1\big\rVert}}

\newcommand{\skpb}[2]{{\big\langle#1,#2\big\rangle}}  

\newcommand{\absb}[1]{\big|#1\big|}

\newcommand{\abss}[1]{\bigg|#1\bigg|}
\newcommand{\Abss}[1]{\Bigg|#1\Bigg|}

\newcommand{\R}{{\mathbb{R}}}
\newcommand{\N}{{\mathbb{N}}}

\newcommand*{\defeq}{\mathrel{\vcenter{\baselineskip0.5ex \lineskiplimit0pt
			\hbox{\scriptsize.}\hbox{\scriptsize.}}}%
	=}
\newcommand*{\defeql}{ = \mathrel{\vcenter{\baselineskip0.5ex   	\lineskiplimit0pt
			\hbox{\scriptsize.}\hbox{\scriptsize.}}}%
}

\newcommand{\wjb}[2]{w_{\bs j}(\bs {#1}; {#2})}  

\newcommand{\wjsb}[3]{w_{\bs j}^{(\bs{#3})}(\bs {#1}; {#2})} 
\newcommand{\wksb}[3]{w_{\bs k}^{(\bs{#3})}(\bs {#1}; {#2})}

\newcommand{\Zi}[2]{\partial^{\bs{#1}}Z_i(\bs {#2})}

\newcommand{\s}{^{(\bs{s})}}
\newcommand{\dels}{\partial^{\bs s}}

\newcommand{\muestsb}[2]{\hat\mu_{n}^{(\bs s)}({\bs #1};{#2})}

\newcommand{\fkmin}{{f_k}_{\min}}
\newcommand{\fkmax}{{f_k}_{\max}}

\newcommand{\Cmax}{C_1}
\newcommand{\Clip}{C_2}
\newcommand{\Ccard}{C_{\mathrm{d}}}
\newcommand{\Csum}{C_3}

\newcommand{\Cenv}{C_{\mathrm{e}}}

\newcommand{\E}{{\mathbb{E}}}

\newcommand{\hclass}{\mathcal{H}(\alpha,L)}

\newcommand{\Pclass}{\mc P(\beta, C_Z)}

\newcommand{\oY}[1]{\bar{Y}_{#1}}

\DeclareMathOperator*{\argmin}{argmin}

\endlocaldefs

\begin{document}

\begin{frontmatter}
\title{Smooth and rough paths in mean derivative estimation for functional data}
\runtitle{Smooth and rough paths in mean derivative estimation}

\begin{aug}
\author[A]{\fnms{Max}~\snm{Berger}\ead[label=e1]{mberger@mathematik.uni-marburg.de}}
\author[A]{\fnms{Hajo}~\snm{Holzmann}\ead[label=e2]{holzmann@mathematik.uni-marburg.de}}
\address[A]{Philipps-University, Marburg, Germany\printead[presep={,\ }]{e1,e2}}
\end{aug}

\begin{abstract}
In this paper, in a multivariate setting we derive near optimal rates of convergence in the minimax sense for estimating partial derivatives of the mean function for functional data observed under a fixed synchronous design over Hölder smoothness classes. 
In contrast to mean function estimation, for derivative estimation the smoothness of the paths of the processes is crucial.  For processes with rough paths of lower-order smoothness than the order of the partial derivative to be estimated, we determine a novel, slower than parametric optimal rate of convergence. For processes with paths of higher-order smoothness  we show that the parametric $\sqrt n$ rate can still be achieved under sufficiently dense design. 
We conduct our analysis in  the supremum norm since it corresponds to the visualisation of the estimation error. Further, as a basis for  the construction of  uniform confidence bands we also derive a central limit theorem in the space of continuous functions equipped with the sup-norm.

Derivative estimation is of quite some interest in functional data analysis, e.g.~to assess the dynamics of the underlying processes. We implement a multivariate local polynomial derivative estimator and illustrate its finite-sample performance in a simulation as well as for two real-data sets. To determine the smoothness of the sample paths in the applications we further discuss a method based on comparing restricted estimates of the partial derivatives of the covariance kernel.  
\end{abstract}


\end{frontmatter}



\section{Introduction}

Estimating derivatives is frequently of interest in functional data analysis (FDA). In their influential book, as motivating examples  \citet{ramsay1998functional} discuss estimation of acceleration (second derivatives) of height curves of children or estimation of the yearly temperature forcing function defined in terms of first and third derivatives, thus stressing the importance of derivative estimation in FDA. 
\citet{wang2016functional} point to the role of derivative estimation for assessing the dynamics of the underlying processes. 

We show in this paper the novel feature that surprisingly, rates of convergence for derivative estimation depend on the smoothness of the paths of the processes. In particular, for processes with rough paths of lower-order smoothness than the order of the partial derivative to be estimated, we determine a novel, slower than parametric optimal rate of convergence. To assess the smoothness of the sample paths in applications we further introduce a method based on comparing restricted estimates of the partial derivatives of the covariance kernel. As it turns out, rough, non-differentiable sample paths seem not to be an exception in observed functional data. 

While estimating the mean function and the covariance kernel themselves  has been intensely investigated in the literature on FDA 
 \citep{yao2005functional, li2010uniform, cai2011optimal, degras2011simultaneous, zhang2016sparse, xiao2020asymptotic, wang2020simultaneous, berger2023dense, mohammadi2024functional}, there is much less methodology and theory available for derivative estimation.  
For asynchronous random design, \citet{sharghi2021mean} propose weighted local linear estimators for the first derivative of the mean function, and obtain rates of convergence slower than $1/\sqrt n$. Similar results are obtained also for the mixed partial derivatives of the covariance kernel and the principle component basis functions in \citet{dai2018derivative}. For synchronous, fixed design \citet{benhenni2014local} consider local polynomial estimators for general derivatives and derive pointwise rates of convergence. 
 \citet{cao2014simultaneous} shows that for smooth processes, in a fixed synchronous design with sufficiently many design points spline estimators may achieve the parametric $\sqrt n$-rate of convergence even under the supremum norm.

In this paper, in a multivariate setting we derive near optimal rates of convergence in the minimax sense for estimating partial derivatives of the mean function for functional data observed under a fixed synchronous design over Hölder smoothness classes. We focus on the supremum norm since it corresponds to the visualisation of the estimation error, and is closely related to the construction of  uniform confidence bands. Similarly to results in \citet{berger2023dense} for the mean function, if the number of design points in each row is too low a discretization term dominates, and an intermediate regime occurs in the rates under the supremum norm in which the observational errors dominate. 

However, in contrast to mean function estimation, for derivative estimation the smoothness of the paths of the processes is crucial for the rates of convergence. On the one hand, if the paths of the processes have higher-order smoothness than the order of the partial derivative to be estimated, the parametric $\sqrt n$ rate can be achieved under sufficiently dense design. On the other hand, for processes with rough paths of lower-order smoothness, we show that the rates of convergence are necessarily slower than the parametric rate, and we determine a near-optimal rate at which estimation is still possible.

The paper is organized as follows. In Section \ref{sec:mainresults}, we start by formally stating the multivariate model and by defining the Hölder-smoothness classes of functions and processes that we will be working with. Section \ref{sec:optimalrates} contains the main results: We provide upper bounds for estimating partial derivatives in FDA using linear nonparametric estimators, which we complement with lower bounds which almost match the upper bounds. We also show asymptotic normality of the derivative estimator under the $\sqrt n$-regime, which requires smooth paths for the processes. Based on the restricted polynomial estimators of the covariance kernel proposed in \citet{berger2024optimal}, in Section \ref{sec:covdersmoothpaths} we further propose a method for assessing the smoothness of the covariance kernel on the diagonal, which is a necessary condition for smoothness of the paths of the processes. In Section \ref{sec:local_polynomial_estimator} we discuss multivariate local-polynomial derivative estimators and  show that the resulting weights fulfil the requirements needed for the upper bounds in Section \ref{sec:optimalrates}. 
Section \ref{sec:sim} contains a simulation study, which is contained in the \href{https://github.com/mbrgr/smooth-rough-paths-fda-derivative-estimation}{GitHub Repository: smooth-rough-paths-fda-derivative-estimation}. Apart from investigating the effects of the choice of the bandwidth and of the number of design points, we illustrate how smooth and rough processes influence the error in  derivative estimation. In Section \ref{sec:real_data_application} we revisit two data examples from \citet{berger2023dense} and complement the analysis by estimating first derivatives.  Section \ref{sec:conclude} concludes, while the main proofs are gathered in Section \ref{sec:proof:theorem:rates}. The Appendix contains further technical arguments in Section \ref{app:additional_proofs} as well as additional simulation results in Section \ref{sec:simfullint}. 

Let us conclude the introduction by summarizing notation used in the reminder of the paper. The dimension is denoted by $d \in \N$. 
For $ \bs a=(a_1,\ldots,a_d)^\top$ and $ \bs b=(b_1,\ldots,b_d)^\top \in \R^d$ we write $\bs a\leq \bs b$ if $a_i\leq b_i$ for $i=1,\ldots,d$. We set  $\bs{a}^{\bs{b}} \defeq a_1^{b_1}\cdots a_d^{b_d}$,  $|\bs a|=a_1+\ldots+a_d$,  $\bs a!\defeq a_1!\cdots a_d!$, $ \bs a_{\min} = \min_{1\leq r \leq d} a_r$. If the coordinates of $\bs a$ are non-negative integers we denote the partial derivative operator by $\partial^{\bs{a}} \defeq \partial_1^{a_1}\ldots\partial_d^{a_d}$. Set $\bs 1\defeq (1,\ldots,1)^\top \in \N^d$. Given $\bs p \defeq (p_1,\ldots,p_d)^\top$ we denote by $\{\bs j \in \N^d \mid \bs 1 \le \bs j \le \bs p\}$ the set of $\bs j \in N^d$ that are component- wise between $\bs 1$ and $\bs p$, and denote $\sum_{\bs j=\bs 1}^{\bs p}\defeq \sum_{j_1=1}^{p_1} \cdots \sum_{j_d=1}^{p_d}$. If for a vector the indices are written in bold letters this indexes vectors such as $\bs{p_j}\defeq (p_{j_1},\ldots,p_{j_d})^\top$. Scalars indexed by vectors and scalars are written as  $\e_{i,\bs j} \defeq \e_{i,j_{1},\ldots,j_{d}}$ or $ Y_{i,\bs j} \defeq Y_{i,j_{1},\ldots,j_{d}}$.   
The supremum norm (sup-norm) on $[0,1]^d$  is denoted by $\norm{\,\cdot\,}_\infty$.

\section{Optimal rates of convergence for derivative estimation}\label{sec:mainresults}

Let us start by introducing the model: The observed data $(Y_{i,  \bs j}, \bs{x_{ j}}), i = 1, \ldots, n, \bs j = \bs 1, \ldots, \bs p$ are distributed according to  
\begin{align}
	Y_{i,  \bs j}= \mu( \bs{x_{ j}}) + Z_i( \bs{x_{ j}}) + \e_{i, \bs j} \,, \quad  i=1,\dotsc,n\,, \bs j = \bs 1,\ldots,\bs p  \,. \label{eq:model}
\end{align}
Here, $Y_{i, \bs j}$ are real-valued response variables and the $\bs{x_j} \in [0,1]^d$ are fixed, known design points which are synchronous over the repetitions $i$. The $Z_1,\dotsc,Z_n$ are independent and identically distributed centered processes with continuous paths, and the observations errors $\e_{i,\bs j}$ are independent with mean zero and also independent of the $Z_i$. Thus $\mu$ is the deterministic mean function.

Our main results, Corollary \ref{cor:rateconvup} and Theorem \ref{thm:lower:bound:mean:derivatives}, imply in particular that if the smoothness of the mean function $\mu$ is $\alpha>0$ and of the paths of the processes $Z_i$ it is $\beta>0$, if the $\bs s$-partial derivative of $\mu$ is to be estimated, and if $\beta < |\bs s | < \alpha$, then $\mu$ can be estimated no faster than with rate   
$$n^{- \frac{\alpha -\abs{\bs s}}{2(\alpha - \beta )}}$$
which appears to be a novel rate, slower than the parametric rate $n^{-1/2}$. 

The section is now organized as follows. In Section \ref{sec:smoothnessclass} we introduce the smoothness classes of functions and processes which we work with. Section \ref{sec:optimalrates} contains our main results: Near optimal rates of convergence for derivative estimation, together with a central limit theorem. While the upper bounds are obtained for generic linear estimators with weights satisfying certain conditions, Section \ref{sec:local_polynomial_estimator} discusses local polynomial weights for multivariate derivative estimation. In Section \ref{sec:covdersmoothpaths} we describe a method for assessing the smoothness of the covariance kernel on the diagonal, which is a necessary condition for smoothness of the paths of the processes.

\subsection{Smoothness classes of functions and processes}\label{sec:smoothnessclass}

For an open, bounded and convex set $U\subset \R^d$ and an integer $m\geq 1$ we define 
\begin{align}
	C^{m}(\overline{U})  \defeq \big\{\,&f\colon U \to \R \mid f\text{ is } m\text{ times continuously differentiable in }U \text{ and for}\\
	& \abs{\bs s} \leq m \text{ the partial derivative } \partial^{\abs{\bs s}} f \text{ can be continuously extended to } \overline{U}\, \big\} \nonumber
\end{align}
	Given $\alpha >0$ we set $\left\lfloor \alpha \right\rfloor=\max\{  k\in\N_0 \mid k<\alpha \}.$  
	A function $f\colon \overline{U} \to \R$ is \textit{H\"older-smooth} with index $\alpha$ if for all indices $\bs \beta=(\beta_1,\ldots,\beta_d)^\top$ with $|\bs \beta|\leq \left\lfloor \alpha \right\rfloor$ the partial derivatives $\partial^{\bs \beta} f(\bs x)$ exist 
	and the H\"older-norm given by
	$$\norm{f}_{\mc H, \alpha}\defeq\max_{|\bs \beta| \leq \left\lfloor \alpha \right\rfloor } \sup_{\bs x \in \overline{U}} |\partial^{\bs\beta} f(\bs x)|+ \max_{|\bs \beta|=\left\lfloor \alpha \right\rfloor}\sup_{\bs x,\bs y \in \overline{U}, \,\bs x\neq \bs y} \frac{|\partial^{\bs \beta} f(\bs x)-\partial^{\bs \beta} f(\bs y)|}{\norm{\bs x-\bs y}_\infty^{\alpha-\left\lfloor \alpha \right\rfloor}}$$
	is finite.  
	The H\"older class on $\overline{U}$ with parameters $\alpha>0$ and $L>0$ is defined by
	\begin{equation}\label{def:hoelder:class}
		\mc H_U(\alpha, L) = \mc H(\alpha, L) = \big\{f \in C^{\lfloor \alpha \rfloor}(\overline{U}) \mid \norm{f}_{\mc H, \alpha} \leq L \big\}.
	\end{equation}
 Note that all partial derivatives of functions $f \in \mc H_U(\alpha, L)$ of order $< \lfloor \alpha \rfloor$ are Lipschitz-continuous with constant bounded by $L$. 

In the framework of our model \eqref{eq:model}, $U=(0,1)^d$ and $\overline{U}=[0,1]^d$. Apart for the mean function $\mu$, we also need to assume some smoothness of the paths of the processes $Z_i$. For $\beta >0$ consider processes $Z\colon[0,1]^d\to \R$ for which partial derivatives of order $\leq \lfloor \beta \rfloor$ exist almost surely, and for which the Hölder norm $\norm{Z}_{\mc H, \beta} $ is almost surely bounded by a square integrable random variable $M$: $\norm{Z}_{\mc H, \beta} \leq M$ a.s.~with $\expec [M^2] < \infty $. Thus for $\beta >0$ and $C_Z>0$ we consider the class of processes 
	\begin{align}
		\mc P(\beta, C_Z) = \mc P(\beta) & = \big\{ Z\colon[0,1]^d\to \R \text{ process} \mid \text{differentiable paths up to order} \leq \lfloor \beta \rfloor, \ \nonumber\\
		&  \exists \, \text{random variable } M \text{ s.th. } \norm{Z}_{\mc H, \beta} \leq M \text{ a.s. and } \ \expec [M^2] \leq C_Z \big\} \label{eq:classprocesses} 
	\end{align}	
Thus processes $Z \in \mc P(\beta, C_Z) $ have Lipschitz-continuous derivatives of order $<  \lfloor \beta \rfloor$, and Hölder-continuous derivatives of orders $= \lfloor \beta \rfloor$ with exponent $\beta -  \lfloor \beta \rfloor$ a.s., and moreover, all random  Lipschitz - and Hölder constants can be upper-bounded by a square-integrable random variable $M$ with $\expec[M^2] \leq C_Z$. 

\begin{example}[Fractional Brownian motion]
    Fractional Brownian motion (fBm) $(B_t^H)_{t \in [0,1]}$ with Hurst parameter $H \in (0,1)$ is a centered Gaussian process with covariance function $\expec[B_t^H B_s^H] = (t^{2H} + s^{2H} - |t-s|^{2H})/2$ \citep{shevchenko2014fractional}. Since the variogram of fBm is $\expec[(B_t^H- B_s^H)^2] = |t-s|^{2H}$, Hölder continuity of the paths of any order $H-\epsilon$, $0 < \epsilon < H$ follows from the Kolmogorov-Chentsov  theorem, where integrability of the Hölder-constant is shown e.g.~in \citet[Theorem 1]{azmoodeh2014necessary}. Smooth processes can be obtained by taking iterated integrals of the fBm. 

    A related example is the Riemann-Liouville version of $\beta$-fractional Brownian motion. It can be directly defined for all $\beta >0$ by 
    \begin{equation}\label{eq:RLfB1}
    R_t^\beta =  \int_0^t (t-s)^{\beta-1/2}\,  \dx W_s,
    \end{equation}
    where $(W_s)$ is a standard Brownian motion. See \citet[Section 10]{van2008reproducing}. In the proof of Theorem \ref{thm:lower:bound:mean:derivatives} in Section \ref{sec:prooflowerbound} we discuss sample path smoothness of this class of processes. 

\end{example}

	%

\subsection{Optimal rates of convergence and asymptotic normality}\label{sec:optimalrates}

 Let us make precise the design assumptions that we require in model \eqref{eq:model}: the design points $\bs x_{\bs j}=(x_{1,j_1}, \ldots, x_{d,j_d})$, $1 \leq j_k \leq p_k$, $k=1, \ldots, d$, and $x_{k,l} < x_{k,l+1}$,  $1 \leq l \leq p_k - 1$ have  Cartesian product  structure. Set $\bs p(n)= \bs p  \defeq (p_1,\ldots,p_d)^\top$. The number $\bs p^{\bs 1} =\bs p^{\bs 1}(n) =\prod_{k=1}^d p_k(n)$ of design points and the design points $\bs x_{\bs j}$ will depend on $n$, which is suppressed in the notation. 
\begin{assumption}[Regular design]\label{ass:design:distribution} 
	There is a constant $\Ccard >0$ such that for each $\bs x \in [0,1]^d$ and $h>0$ we have that 
	$$\card \big\{\bs j \in \{\bs 1,\ldots, \bs p\} \mid \bs x - (h, \ldots, h) \leq \bs{x_j} \leq \bs x + (h, \ldots, h)\big\}  \leq \Ccard\, h^d\,\bs p^{\bs 1}. $$ 
\end{assumption}



	Turning to the estimation of the   partial derivative $\dels \mu$  of the mean function $\mu$, we consider linear estimators
	\begin{equation}\label{eq:linest:mu:bar}
		\muestsb{x}{h}  =  \sum\limits_{\bs j= \bs 1}^{\bs p}   w_{\bs j,\bs p}^{(\bs s)}(\bs x;h; \bs x_{\bs 1},\dotsc,\bs x_{\bs p}) \, \bar Y_{\bs j} \,,  \quad \bar Y_{\bs j} = \frac1n \sum_{i=1}^nY_{i,\bs j},
	\end{equation}
	where $\bs x \in [0,1]^d$,  $w^{(\bs s)}_{\bs j,\bs p}(\bs x;h; \bs x_{\bs 1},\dotsc,\bs x_{\bs p}) = \wjsb xhs$ are deterministic weights depending on the design points, on a bandwidth parameter $h >0$ and the order $\bs s$ of the derivative to be estimated. 


\smallskip

Next we list assumptions used for the weights which we check for local polynomial weights in Section \ref{sec:local_polynomial_estimator}.

\begin{assumption}[Weights of linear estimator]\label{ass:weights} 
For a certain $\bs s \in \N^d$ there exist $c>0$ and $h_0>0$ such that for  sufficiently large $\bs p_{\min} $,  the following holds for all $h \in (c/\bs p_{\min}, h_0]$ for suitable constants $\Cmax, \Clip>0$ which are independent of $n,\bs p,h$ and $\bs x$.
\begin{enumerate}[label=\normalfont{(M\arabic*)},leftmargin=12.5mm]
	\item \label{ass:weights:pol:rep} The weights reproduce the $\bs s^{\mathrm{th}}$ derivative of polynomials of degree $\zeta \geq \abs{\bs s}$. For a polynomial $Q$ of degree $\zeta$ we have
	$$ \hat \mu^{(\bs s)}_{n,\bs p, h}(\bs x; h; \{Q(\bs{x_j})\}_{\bs j = \bs 1,\ldots, \bs p}) = \dels Q(\bs x)\,, \bs x \in [0,1]^d, $$
	or equivalently
	\begin{equation}
		\sum_{\bs 1 \leq \bs j \leq \bs p}(\bs x_{\bs j}-\bs x)^{\bs r} \, \wjsb xhs = \delta_{\bs r,\bs s} \bs s!\,, \quad \bs r \in \N^d \text{ s.t. } \, \abs{\bs r} \leq \zeta \,, \,\bs x \in [0,1]^d. \label{eq:weights:rep:pol:analytic}
	\end{equation}
	\item \label{ass:weights:equal:0} \textit{(Localize).} \quad We have $\wjsb xhs = 0$ if $\norm{\bs x- \bs{x_j}} > h$ for $\bs x \in [0,1]^d$. 
	\item \textit{(Boundedness).} \quad For the absolute values of the weights $$\max_{\bs 1\leq \bs j \leq \bs p} \abs{\wjsb xhs}\leq \frac{\Cmax}{\bs{p^1}\,h^{d+\abs{\bs s}}}, \;\bs x\in [0,1]^d\,.$$
	\label{ass:weights:max}
	\item \textit{(Lipschitz continuous).} \quad It holds
	$$ \absb{ \wjsb xhs - \wjsb yhs } \leq \frac{\Clip}{\bs{p^1}\,h^{d +\abs{\bs{s}}}} \bigg( \frac{\norm{\bs x- \bs y}_\infty }{h} \vee 1 \bigg)\,, \quad \bs x, \bs y \in [0,1]^d\,. $$ \label{ass:weights:lipschitz}
\end{enumerate}	
\end{assumption}
%



   
\begin{assumption}[Sub-Gaussian errors] \label{ass:model} 
		The random variables $\{\e_{i,\bs j} \mid 1 \leq i \leq n,\, \bs 1 \le \bs j \le \bs p\}$ are independent and independent of the processes $Z_1,\dotsc,Z_n$. Further we assume that the distribution of $\e_{i,\bs j}$ is sub-Gaussian, and setting $ \sigma_{i, \bs j}^2 \defeq  \E[\e_{i,\bs j}^2]$ we have that $\sigma^2 \defeq \sup_i \max_{\bs 1 \leq \bs j \leq \bs p} \sigma_{i,\bs j}^2 < \infty$ and that there exists $\zeta>0$ such that $\zeta^2\sigma_{i,\bs j}^2$ is an upper bound for the  sub-Gaussian norm of $\e_{i,\bs j}$. \label{ass:errors} 
\end{assumption}

\begin{theorem} [Upper bounds]\label{thm:mean:upper_bounds} 
	Consider model \eqref{eq:model} under Assumptions \ref{ass:design:distribution} and \ref{ass:model} with $\mu \in \mc H(\alpha, L)$ and $Z\in \mc P(\beta, C_Z)$ for $\alpha, \beta, L, C_Z>0$. Given $\bs s \in \N_0^d$ with $\abs{\bs s} \leq \floor{\alpha}$ consider the linear estimator $\hat \mu_{n,\bs p,h}^{(\bs s)}(\bs x)$ of the derivatives $\dels \mu$ of $\mu$ in \eqref{eq:linest:mu:bar}, and suppose that the weights satisfy Assumption \ref{ass:weights} with $\zeta =\floor \alpha$. Then for sufficiently large $\bs p_{\min}$ and $n$ we have that 
	\begin{equation}\label{eq:upperboundone:w:deriv}
		\sup_{h \in (c/{\bs p_{\min}}, h_0]} \sup_{Z \in \mc P(\beta, C_Z),\, \mu \in \mc H(\alpha,L)}\, a_{n,\bs p, h}^{-1}\, \expec_{\mu,Z} \Big[ \normb{ \hat \mu_{n,\bs p,h}^{(\bs s)}- \dels\, \mu}_\infty \Big] = \cO(1),
	\end{equation}
	with $c, h_0 >0 $ as in Assumption \ref{ass:weights}, and where 
	\begin{equation}\label{eq:orderbound:w:deriv}
		a_{n,\bs p, h} = \max\bigg( L\, h^{\alpha - \abs{\bs s}}, \zeta\, \sigma\, \sqrt{\frac{\log(h^{-1})}{n\bs{p^1}h^{d+2\abs{\bs s}}}}, \frac{C_Z}{\sqrt n\, h^{\abs{\bs s}- \kappa }}\bigg)\,,  
	\end{equation}
    	with $\kappa \defeq \min(\abs{\bs s}, \beta_1)$ for any fixed $ \beta_1$ with $ 0 < \beta_1 < \beta$. 	
\end{theorem}

Theorem \ref{thm:mean:upper_bounds}  follows from Lemmas \ref{lem:mean:estimation:rates:asbefore} and \ref{lem:mean:estimation:rates:new} in Section \ref{sec:proof:theorem:rates}.

\begin{corollary}[Rates of convergence]\label{cor:rateconvup}
    Under the assumptions of Theorem \ref{thm:mean:upper_bounds} we have the following. For the smoothness $\beta>0$ of the processes $Z\in \mc P(\beta, C_Z)$, 
    \begin{enumerate}
    \item \textit{(Rough processes)} 
    if $\beta \leq \abs{\bs s}$  then for any fixed $0 < \beta_1 < \beta$, setting
	\begin{align}
		h^\star \simeq \max \bigg\{ \frac1{\bs{p}_{\min}}, \bigg( \frac{\log(n \bs{p^1})}{n\,\bs{p^1}}\bigg) ^{ \frac1{2\alpha + d} },n^{- \frac{1}{2(\alpha - \beta_1)}}  \bigg\}\,,\label{eq:deriv:h:star:kappa}
	\end{align}
	we obtain 
   	 \end{enumerate}
    {\small 
	\begin{align}
		& \sup_{Z \in \mc P(\beta),\, \mu \in \mc H(\alpha)}\,  \expec_{\mu,Z} \Big[ \normb{ \hat \mu_{n,\bs p,h^\star}^{(\bs s)} - \dels\, \mu}_\infty \Big] 
        = & \mc O \bigg(\max \bigg(\bs p_{\min}^{-(\alpha -\abs{\bs s})},\bigg( \frac{\log(n \bs{p^1})}{n\,\bs{p^1}}\bigg) ^{ \frac{\alpha -\abs{\bs s}}{2\alpha + d} },  n^{- \frac{\alpha -\abs{\bs s}}{2(\alpha - \beta_1 )}}  \bigg)\bigg)\,, \label{eq:deriv:rate:kappa}
	\end{align}
    }
      \begin{enumerate}
    \item[ii)] \textit{(Smooth processes)}  
    if $\beta > \abs{\bs s}$, then setting 
	\begin{align}
		h^\star \simeq \max \bigg\{ \frac1{\bs{p}_{\min}}, \bigg( \frac{\log(n \bs{p^1})}{n\,\bs{p^1}}\bigg) ^{ \frac1{2\alpha + d} } \bigg\}\,,\label{eq:deriv:h:star}
	\end{align}
	we obtain 
   	 \end{enumerate}
	\begin{align}
		\sup_{Z \in \mc P(\beta),\, \mu \in \mc H(\alpha)}\,  \expec_{\mu,Z} \Big[ \normb{ \hat \mu_{n,\bs p,h^\star}^{(\bs s)} - \dels \mu}_\infty \Big] = \mc O \bigg(\max \bigg(\bs p_{\min}^{-(\alpha -\abs{\bs s})},\bigg( \frac{\log(n \bs{p^1})}{n\,\bs{p^1}}\bigg) ^{ \frac{\alpha -\abs{\bs s}}{2\alpha + d} }, n^{ - \frac12}  \bigg)\bigg)\,. \label{eq:deriv:rate}
	\end{align}

        
\end{corollary}
This follows from Theorem \ref{thm:mean:upper_bounds}: The second term in the maxima of \eqref{eq:deriv:h:star} and \eqref{eq:deriv:h:star:kappa} balances the first two terms in \eqref{eq:orderbound:w:deriv}, while the third term in \eqref{eq:deriv:h:star:kappa} balances the first and third terms in \eqref{eq:orderbound:w:deriv} for $\kappa= \beta_1 < \abs{\bs s}$. The Assumption \ref{ass:weights} requires the bandwidth to be of order at least $1/\bs p_{\min}$.

\begin{remark}[Improving the bound \eqref{eq:deriv:rate:kappa}] 

We point out that the last term in the rate \eqref{eq:deriv:rate:kappa} can be improved to obtain
$$\Big(\frac{\log n}{n}\Big)^{ \frac{\alpha -\abs{\bs s}}{2(\alpha - \beta )}},$$
so that
	{\small 
    \begin{align*}
		& \sup_{Z \in \mc P(\beta),\, \mu \in \mc H(\alpha)}\,  \expec_{\mu,Z} \Big[ \normb{ \hat \mu_{n,\bs p,h^\star}^{(\bs s)} - \dels\, \mu}_\infty \Big] 
        = & \mc O \bigg(\max \bigg(\bs p_{\min}^{-(\alpha -\abs{\bs s})},\bigg( \frac{\log(n \bs{p^1})}{n\,\bs{p^1}}\bigg) ^{ \frac{\alpha -\abs{\bs s}}{2\alpha + d} },  \Big(\frac{\log n}{n}\Big)^{ \frac{\alpha -\abs{\bs s}}{2(\alpha - \beta )}} \bigg)\bigg)\,. 
	\end{align*}
    }
 This follows from the bound \eqref{eq:bound:rougher:paths:then:derivative} in Lemma \ref{lem:mean:estimation:rates:new} in Section \ref{sec:proof:theorem:rates} by taking $\delta = 1/\log n$ in \eqref{eq:bound:rougher:paths:then:derivative}, and using $(\log n / n)^{- \frac{1}{2(\alpha - \beta)}}$ as the final term in the maximum \eqref{eq:deriv:h:star:kappa} for $h^*$. 
However, for classes of processes such as fractional Brownian motion this does not constitute a meaningful improvement since there is no maximal exponent at which these are Hölder continuous. 
    
\end{remark}

\begin{remark}[Effect of smoothness of paths, and parametric rate]
As for the mean function, for the parametric rate  $1/\sqrt n$ to be attainable requires a sufficiently large number of design points, that is a sufficiently large $\bs p_{\min}$, so that the contribution of the discretization error, $\bs p_{\min}^{-(\alpha -\abs{\bs s})}$, as well as that of the observational noise, the second term in \eqref{eq:deriv:rate} is  negligible compared to $1/\sqrt n$. 

However, estimating the partial derivative of order $\abs{\bs s}$ of the mean function $\mu$ at the parametric rate also requires sufficiently smooth paths for the processes $Z_i$. If the paths themselves are smooth of order $\beta < \abs{ \bs s}$, then on the one hand consistent estimation is still possible, but on the other hand the parametric rate can no longer be achieved, but only a rate of approximately
$n^{- (\alpha -\abs{\bs s})/(2(\alpha - \beta ))}.$
We will complement this by providing lower bounds below.     
\end{remark}

\begin{remark}[Relation to nonparametric regression with dependent errors]
After forming means over the observational rows, model \eqref{eq:model} reduces to single row
\begin{align}\label{eq:reducednonreg}
	\bar Y_{  \bs j,n}= \mu( \bs{x_{ \bs j}}) + \bar{Z}_{n}( \bs{x_{ \bs j}}) + \bar \epsilon_{\bs j,n} \,, \quad   \bs j = \bs 1,\ldots,\bs p  \,, 
\end{align}
where
\begin{align*}
	\bar Y_{\bs j,n} = \frac1n \sum_{i=1}^n Y_{i,\bs j}, \quad \bar \epsilon_{\bs j,n} = \frac1n \sum_{i=1}^n \varepsilon_{i,\bs j} \qquad \text{and} \qquad \bar{Z}_{n}(\bs x) = \frac1n \sum_{i=1}^n Z_i(\bs x) \, .
\end{align*}
While \eqref{eq:reducednonreg} can be interpreted as a nonparametric regression model with depend errors, consisting of the sum $\bar{Z}_{n}( \bs{x_{ j}}) + \bar \epsilon_{\bs j,n}$, note that the covariance $\text{cov}(\bar{Z}_{n}( \bs{x_{ j}}), \bar{Z}_{n}( \bs{x_{ k}}))$ need to become small even if $|\bs j - \bs k|$ is large. Indeed, $\bs p$ does not effect the convergence rate of the error term involving the processes. Therefore, there is no close relation to the literature on nonparametric regression with long-range dependence as e.g.~\citet{hall1990nonparametric}.  
\end{remark}

Next we turn to lower bounds. For the first bound we use the following design assumption, which we also require  in the construction of the local polynomial weights.  
\begin{assumption} \label{ass:designdensity}
	Assume that there exist Lipschitz continuous densities $f_k\colon[0,1]\to \R$ bounded by $0<\fkmin \leq f_k(t) \leq \fkmax < \infty$, $t \in [0,1]$ and $1 \leq k \leq d$, such that the design points $x_{k,l}$,  $1 \leq l \leq p_k $
	\begin{equation*}
		\int_0^{x_{k,l}} f_k(t) \, \dx t=\frac{l-0.5}{p_k} \,, \qquad l=1,\ldots,p_k,\quad k=1,\dotsc,d \,.	
	\end{equation*}
	\end{assumption}
	\citet[Lemma 7]{berger2023dense} shows that Assumption \ref{ass:designdensity} implies Assumption \ref{ass:design:distribution}.
\begin{theorem}[Lower bounds]\label{thm:lower:bound:mean:derivatives}
	Assume that in model \eqref{eq:model} the errors $\epsilon_{i,\bs j}$ are independent and normally $\mc N(0, \sigma_0^2), \sigma_0^2>0$ distributed. Suppose that $\alpha, \beta, L, C_Z>0$ and that $\bs s$ satisfies $\abs{\bs s} \leq \floor{\alpha}$.  
\begin{enumerate}
    \item If the design satisfies Assumption \ref{ass:designdensity}, and all coordinates of $\bs p$ are of the same order, that is $\bs p_{\min} \simeq \bs p_{\max}$, then
	\begin{align}\label{eq:lowerbound1}
		\liminf_{n,\bs p \to \infty} \,\inf_{\hat \mu^{(\bs s)}_{n,\bs p}}\sup_{\substack{\mu \in \mc H(\alpha, L)\\ Z \in \mc P(\beta, C_Z)}} \expec_{\mu, Z}\big[ \bs p_{\min}^{(\alpha -\abs{\bs s})} \norm{\hat \mu_{n,\bs p}^{(\bs s)} - \dels\, \mu}_\infty] \geq c > 0\,.
	\end{align}
    \item Setting 
	$$ a_{n,\bs p} = \max \bigg(\bigg( \frac{\log(n \bs{p^1})}{n\,\bs{p^1}}\bigg) ^{ \frac{\alpha -\abs{\bs s}}{2\alpha + d} }, n^{-1/2} \bigg) \,,$$
	we have 
	\begin{align}\label{eq:lowerbound2}
		\liminf_{n,\bs p \to \infty} \,\inf_{\hat \mu^{(\bs s)}_{n,\bs p}}\sup_{\substack{\mu \in \mc H(\alpha, L)\\ Z \in \mc P(\beta, C_Z)}} \expec_{\mu, Z}\big[ a_{n,\bs p}^{-1} \norm{\hat \mu_{n,\bs p}^{(\bs s)} - \dels\, \mu}_\infty\big] \geq c > 0\,.
	\end{align}
    \item Moreover, if $\beta <  \abs{\bs s}$, then for each $\beta < \beta_1 < \alpha$, we have 
    	\begin{align}\label{eq:lowerbound3}
		\liminf_{n, p \to \infty} \,\max_{\abs{\bs k} = \abs{\bs s}} \inf_{\hat \mu^{(\bs k)}_{n, p}}\sup_{\substack{\mu \in \mc H(\alpha, L)\\ Z \in \mc P(\beta, C_Z)}} \expec_{\mu, Z}\big[ n^{ \frac{\alpha - \abs{\bs s}}{2(\alpha - \beta_1 )}} \norm{\hat \mu_{n, p}^{(\bs k)} -\partial^{\bs k}\mu}_\infty] \geq c > 0\,.
	\end{align}
     
\end{enumerate}
 %
\end{theorem}

\begin{remark}[Discussion of lower bounds]
    The lower bounds in \eqref{eq:lowerbound1} and \eqref{eq:lowerbound2} correspond to the upper bounds in \eqref{eq:deriv:rate}. However, in \eqref{eq:lowerbound1} since the partial derivatives in $\bs s$ are potentially in directions with a higher number of design points as compared to $\bs p_{\min}$, for a consistent statement we require that the all coordinates of $\bs p$ are of the same order.  The rate in \eqref{eq:lowerbound3}  corresponds to the last term in \eqref{eq:deriv:rate:kappa}. However, we have to take exponents $\beta_1$ larger than the critical exponent $\beta$. Further, we only achieve the lower bound if all partial derivatives are taken in the same coordinate, which we incoporate in the statement by considering estimates of all partial derivatives of a given order jointly. Note that these are no additional restrictions in the one-dimensional setting $d=1$.  

    %
\end{remark}

Finally we show asymptotic normality for smooth processes under the $\sqrt n $ - rate. 

    	\begin{theorem}[Asymptotic normality] \label{thm:asymptotic:normality}
        As in Theorem \ref{thm:mean:upper_bounds}, consider model \eqref{eq:model} under the Assumptions \ref{ass:design:distribution} and \ref{ass:model} with $\mu \in \mc H( \alpha,L)$ and $Z\in \mc P(\beta, C_Z)$ for $\alpha, \beta, L, C_Z>0$. Given $\bs s \in \N_0^d$ with $\abs{\bs s} \leq \floor{\alpha}$ let  $\hat \mu_{n,\bs p,h}^{(\bs s)}(\bs x)$ be  the linear estimator of the parial derivative $\dels \mu$ of $\mu$ in \eqref{eq:linest:mu:bar}. 
        
        Suppose that the weights satisfy Assumption \ref{ass:weights} with $\zeta =\floor \alpha$, and that $\beta > \abs{\bs s}$. 
        %
        If  for some $\delta > 1$ and $c_1, c_2>0$ the sequence of intervals 
%
        $$H_n \defeq \big[c_1\,\log(\bs{p^1})^{\delta/(d+2{\abs{\bs s}})}\,\max\big((\log(\bs{p^1})/\bs{p^1})^{1/(2\abs{\bs s} + d)}, \bs{p}_{\min}^{-1}\big), \, c_2\,\log(\bs{p^1})^{-\delta} n^{-\frac1{2(\alpha - \abs{\bs s})}}\big]$$
    is non-empty, than  for smoothing parameters $h_n = h \in H_n $  
        we obtain the CLT
		\begin{align*}
			\sqrt n \big( \muestsb \cdot h -\dels \mu\big) \ \stackrel{D}{\longrightarrow} \ \mc G \big(0,\,\expec[\dels Z\,\dels Z]\,\big)\,,
		\end{align*}
		where $\mc G$ is a real-valued Gaussian process on $[0,1]^d$ with covariance kernel 
        $$\expec[\dels Z(\bs x)\,\dels Z(\bs y)] = \partial^{(\bs s, \bs s)^\top} \Gamma(\bs x, \bs y) 
        \qquad  \bs x, \bs y\in [0,1]^d$$ of the $\bs s$ - partial derivative process $\dels Z$. 
	\end{theorem}
For the proof we check the conditions of the functional central limit Theorem in \citet[Theorem (10.6)]{pollard1990empirical}	for the triangular array of processes $(X_{ni})_{n \in \N, 1 \leq i \leq n}$ given by \eqref{eq:X:process:deriv}. While this is broadly analogous to the case of the mean function $\mu$ in \citet{berger2023dense}, some extensions are required and we provide the proof in Section \ref{sec:proof:theorem:rates}.


\subsection{Assessing the smoothness of the sample paths}\label{sec:covdersmoothpaths}

The central limit theorem, Theorem \ref{thm:asymptotic:normality} forms the basis for the construction of uniform confidence bands for derivatives, which can e.g.~be used to assess  signs of derivatives and hence monotonicity properties of the function itself. 
To achieve the $\sqrt n$-rate of convergence and for Theorem \ref{thm:asymptotic:normality} to be applicable, we require smooth sample paths. 

A necessary condition for differentiable sample paths is the differentiability of the covariance kernel $\Gamma$ on the diagonal \citep[Theorem 8]{da2023sample}.  Restricting to the one-dimensional case $d=1$, consider the upper and lower triangular domains $T_{u}=\{(x,y) \in [0,1]^2 \mid x \leq y\}$ and $T_{l}=\{(x,y) \in [0,1]^2 \mid x \geq y\}$ of $[0,1]^2$, and let  $\partial^{(0,1)}_{u/l}\Gamma$ and $\partial^{(1,0)}_{u/l}\Gamma$ denote the partial derivatives of the covariance kernel $\Gamma(x,y) = \expec[Z_i(x)\, Z_i(y)]$ when restricted to $T_u$ respectively $T_l$, assuming these exist. From the symmetry $\Gamma(x,y) = \Gamma(y,x)$ it follows that we always have (even if $\Gamma$ is not differentiable on the diagonal) that $\partial^{(0,1)}_{u}\Gamma(x,x) = \partial^{(1,0)}_{l}\Gamma(x,x) $. Differentiability on the diagonal requires $\partial^{(1,0)}_{l}\Gamma(x,x) = \partial^{(1,0)}_{u}\Gamma(x,x) $, or 
$$ \partial^{(0,1)}_{u}\Gamma(x,x) = \partial^{(1,0)}_{u}\Gamma(x,x). $$
The method based on local polynomials proposed in \citet{berger2024optimal} to estimate $\Gamma_{|T_u}$ also gives estimates of the partial derivatives $\partial^{(0,1)}_u\Gamma$ and $\partial^{(1,0)}_u\Gamma$. The methods are from the R-package \href{https://github.com/mbrgr/biLocPol}{biLocPol}.


While a formal test for equality of these functions is beyond the scope of the present paper, the estimates may be compared to assess their equality and hence the differentiability of the sample paths.

\subsection{Local polynomial estimators}\label{sec:local_polynomial_estimator}

We briefly discuss local polynomial estimators of partial derivatives. Based on the analysis in \citet[Section 4]{berger2023dense} we show that in case of design densities and mild assumptions on the kernel, the weights satisfy the properties in Assumption \ref{ass:weights}. 

Let $N_{l,d}\defeq {{d+l}\choose{d}}$ and $ \bs \psi_l \colon \{1, \ldots, N_{l,d-1}\} \to \{\bs r \in \{0,1,\ldots, l\}^d \mid |\bs r| = l\} $ be an enumeration of the set $\{\bs r \in \{0,1,\ldots,l\}^d \mid |\bs r | = l\}$ and let $U_m\colon \R^d \to \R^{N_{m,d}}$ be defined as 
$$ U_m(\bs u)\defeq \big(1, P_{\bs \psi_1(1)}(\bs u), \ldots,P_{\bs \psi_{1}(d)}(\bs u), \ldots, P_{\bs \psi_m(1)}(\bs u), \ldots, P_{\bs \psi_m(N_{m,d-1})}(\bs u)\big)$$
for the monomial $P_{\bs \psi_l(j)}(\bs u)\defeq \bs u^{\bs \psi_l(j)}/\bs \psi_l(j)!$, for all $ j=1,\ldots, N_{l,d-1}$ and $ l= 1, \ldots, m$. As a result of the properties of the binomial coefficient we note that $U_m(\bs u) \in \R^{N_{m,d}}$.

Choose a Lipschitz-continuous \textit{kernel function} $K$ with support in $[-1,1]^d$ which satisfies 
\begin{align}
	K_{\min}\one_{[-\bs \Delta, \bs \Delta]}(\bs u) \leq K(\bs u) \leq K_{\max} \qquad \bs u \in \R^d \,,\label{eq:kernel}
\end{align}
for positive constants $\Delta, K_{\min}, K_{\max}>0$, where  $\bs \Delta = (\Delta,\ldots, \Delta)^\top \in \R^d$.
For a \textit{bandwidth} $h >0$, set $K_{h}(\bs x) \defeq K(\bs x/h)$ and $U_{m,h}(\bs x) = U_m(\bs x/h)$.  Given  $m \in \N$ define
\begin{align} \label{eq:locpol} 
	\widehat{\bs \vartheta}_{n,\bs p,h,m}(\bs x) \defeq {\displaystyle\argmin_{\bs \vartheta \in \R^{N_{m,d}}}} \sum\limits_{\bs j=\bs 1}^{\bs p} \big(\oY{ \bs j} - \bs \vartheta^\top\, U_{m,h}(\bs x_{\bs j}-\bs x) \big)^2 K_{h}(\bs x_{\bs j}-\bs x) \,, \quad \bs x \in[0,1]^d\,.
\end{align}
Then for $m \geq \abs{\bs s}$ the \textit{local polynomial estimator of order $m$} of the $\bs s^{\mathrm{th}}$ partial derivative $\dels \mu$ of $\mu$ at $\bs x$ is given by
\begin{align}\label{eq:lp:mu:s}
	\hat \mu_{n,\bs p,h}^{(\bs s), m}(\bs x) =  \skpb{\widehat{\bs \vartheta}_{n,\bs p,h,m}(\bs x)}{\dels U_m(\bs 0)} h^{-\abs{\bs s}}\,. 
\end{align}
where $\dels U_m(\bs 0)$ is the unique vector with a $1$ at the coordinate corresponding to the monomial $\bs u^{\bs s}/\bs s!$.

The proof of the following lemma is provided in Section \ref{app:additional_proofs}.
\begin{lemma}\label{lemma:locpol:weights}
Suppose that the kernel $K$ satisfies \eqref{eq:kernel}. Under the design Assumption \ref{ass:designdensity},  
there exist $p_0 \in \N$ and $h_0, c>0$ such that for $\bs p = (p_1,\ldots,p_d)^\top$ with $\bs p_{\min} \geq p_0$ and $h\in(c/\bs p_{\min},h_0]$, the estimator in \eqref{eq:lp:mu:s} is uniquely defined by \eqref{eq:lp:mu:s} and a linear estimator with weights satisfying Assumption \ref{ass:weights} with $\zeta = m$ for all $\bs s = (s_1, \ldots, s_d)$ with $\abs{\bs s} \leq m$.
\end{lemma}

\section{Simulation and real data examples}

All simulations and the real data exmaples can be found in the \href{https://github.com/mbrgr/smooth-rough-paths-fda-derivative-estimation}{GitHub Repository: smooth-rough-paths-fda-derivative-estimation}.

\subsection{Simulation}\label{sec:sim}

We conduct a simulation in model \eqref{eq:model} in the univariate setting with equidistant design points.   
As mean function we consider 
\begin{equation}
    \mu(x) = \sin\big(3\pi(2x-1)\big)\exp\big(-2(2x-1)^2\big), \qquad x\in [0,1],\label{eq:sim_mu0}
\end{equation}
for which
\begin{align*}
	\mu^\prime (x) & = \Big(6\pi\,\cos\big(3\pi(2x-1)\big) - 4\,(2x-1) \,\sin\big(3\pi\,(2x-1)\big) \Big)\exp\big( -2(2x-1)^2\big)\,. 
\end{align*}

\begin{figure}
	\begin{subfigure}{0.49\linewidth}
		\includegraphics[width=\linewidth]{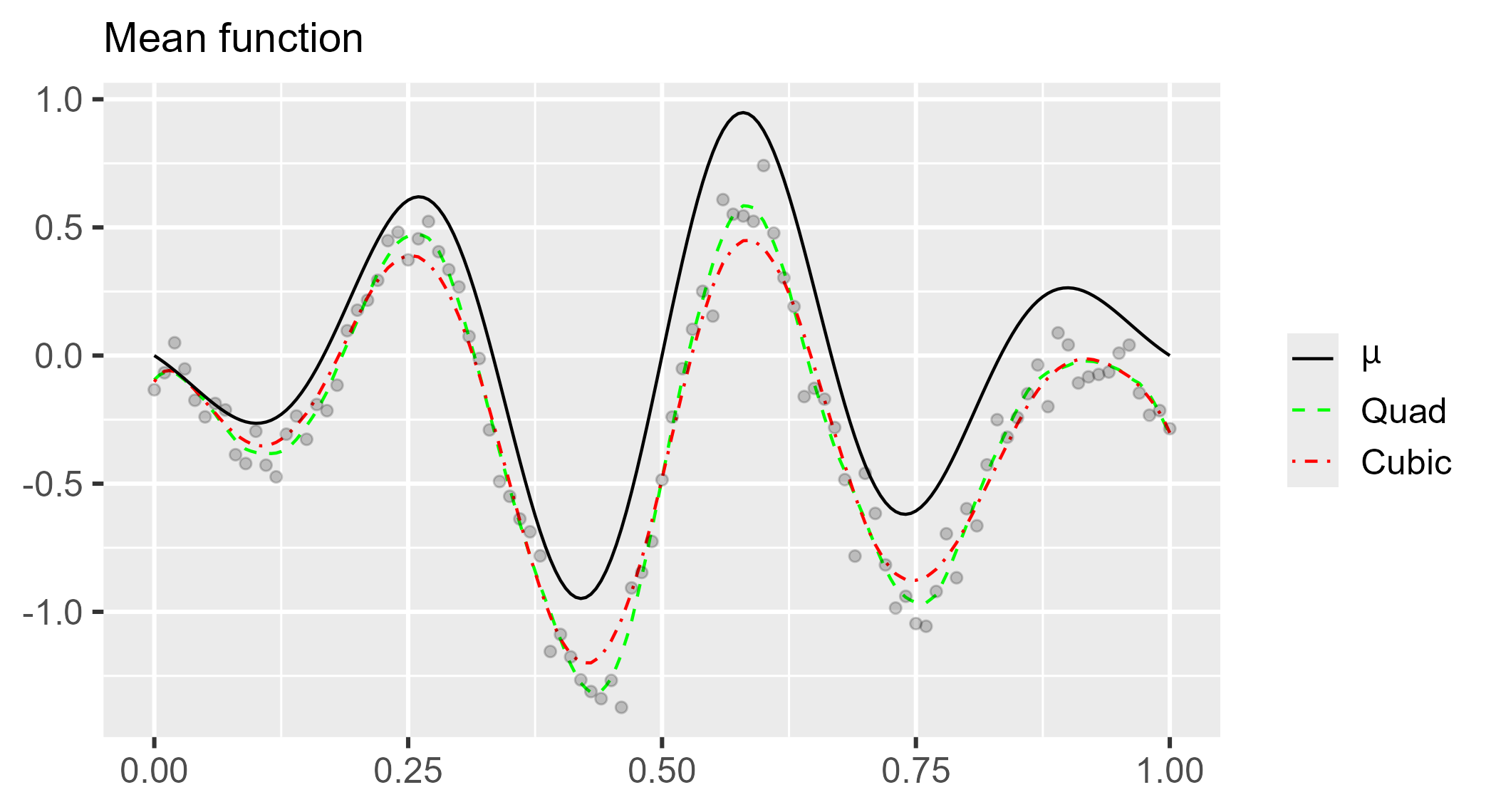}
		\caption{Estimation of the mean function itself.}
		\label{fig:mu_illustration}
	\end{subfigure}
	\hfill
	\begin{subfigure}{0.49\linewidth}
		\includegraphics[width=\linewidth]{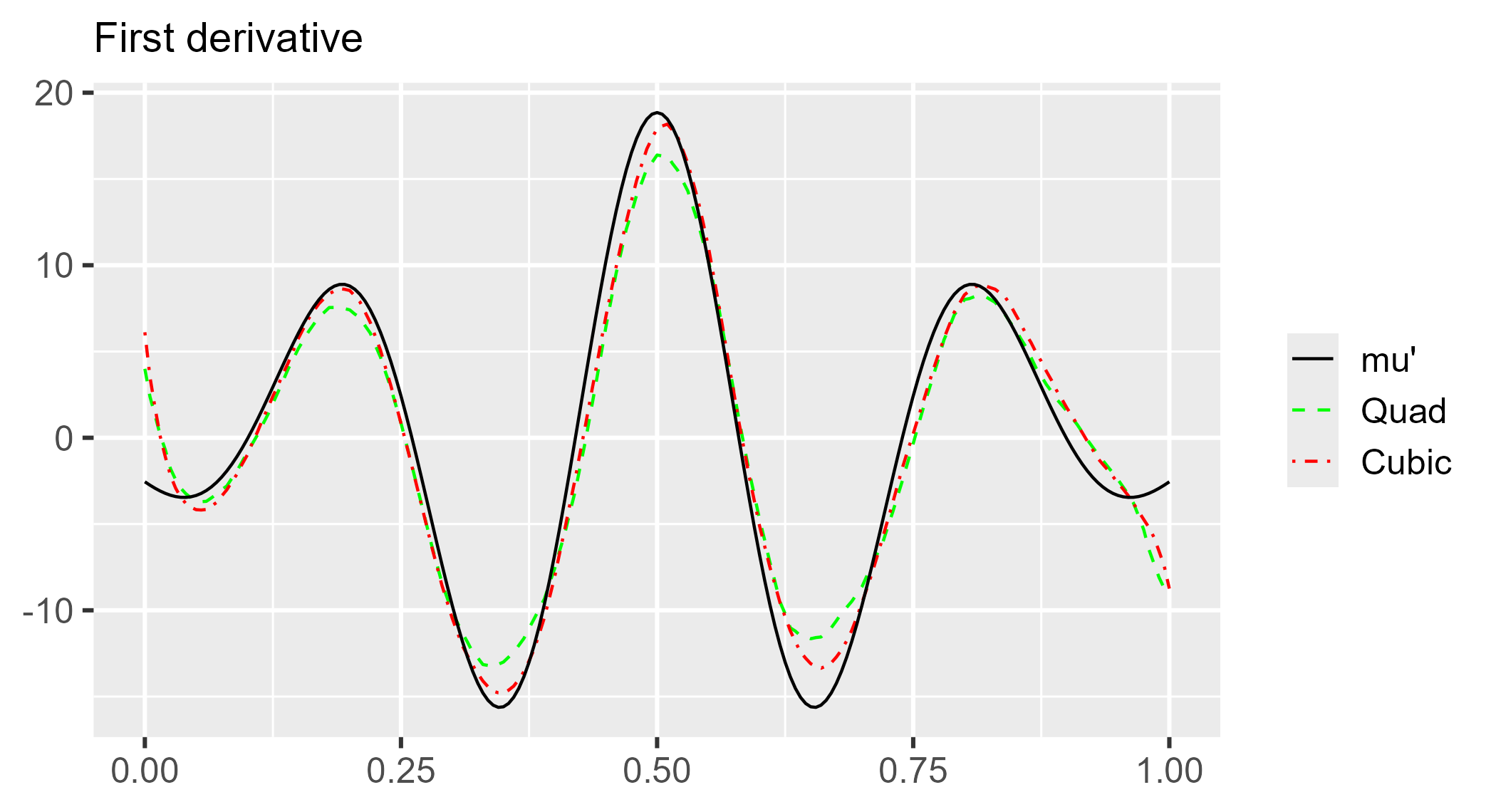}
		\caption{Estimation of the first derivative.}
		\label{fig:mu_prime_illustration}
	\end{subfigure}%
	\caption{The functions $\mu$ (a), its first derivative $\mu^\prime$ (b), the local quadratic estimator (green, lined), the local cubic estimator (red, dot-line) and the observation (black dots). The parameters here are $n = 10$, $p = 101$ and independent errors with $\epsilon_i \sim \mc N(0, 0.09)$.}
	\label{fig:mu_and_mu_prime}
\end{figure}


Figure \ref{fig:mu_and_mu_prime} contains plots of $\mu$ (left figure) and $\mu^\prime$ (right figure), together with local quadratic and local cubic estimates of both $\mu$ and $\mu^\prime$, based on $p = 101$ design points, $n=10$,  $\epsilon_i \sim \mc N(0, 0.09)$ and standard Brownian motion as processes. 
We note that overall, both estimates perform quite reasonably; note the distinct scales of the plots. However the estimators of the derivative have a finite-sample  boundary effect, which does not occur in the estimation theory. This finite-sample effect has been previously observed in nonparametric derivative estimation with local polynomials e.g.~in \citet[Figures 2, 4]{de2013derivative} or in \citet{newell2007comparative}.

Therefore in the following we compute the sup-norm error only over the interval $[h, 1-h]$, where $h$ denotes the bandwidth used in the local polynomial estimator. Corresponding,  qualitatively similar simulation results over the full interval $[0,1]$ are given in Section \ref{sec:simfullint}. 

First we investigate the effect of the choice of the bandwidth and the degree of the local polynomial estimator in dependence on $n$ and $p$. We use independent Brownian motions for the processes $Z_i$ and  independent normal  errors with a standard deviation of $0.5$. First we generate $ N = 1000$ repetitions of $n = 600$ random curves observed at $p \in \{115, 175, 275, 400, 550,$ $ 1000\}$ points in the interval $[0,1]$. For each $p$ we compute the sup-norm error for a grid of bandwidths.   

Figure \ref{fig:deriv_bw_comp_quad} contains the sup-norm errors for the local quadratic estimator and Figure \ref{fig:deriv_bw_comp_cubic} for the local cubic estimator. The local cubic estimator has superior performance, which is in line with results of \cite{fan1996local}, and also depends less sensitively on the choice of the bandwidth. For both estimators, the sup-norm error increases if the bandwidth becomes too small (as well as too large). Figures \ref{fig:error_decomp_n600} and \ref{fig:error_decomp_p100} present simulation results for the components of the error decompositions for different $n$ and $p$, computed with an optimal bandwidth as determined by a grid search. The remainder of the setting is as before. Comparing the sup-error for $n = 600$ and $p = 300$ shows that in this situation more observation points could lower the sup-error more effectively than more observations. 


\begin{figure}
	\begin{subfigure}{0.49\linewidth}
		\includegraphics[width=\linewidth]{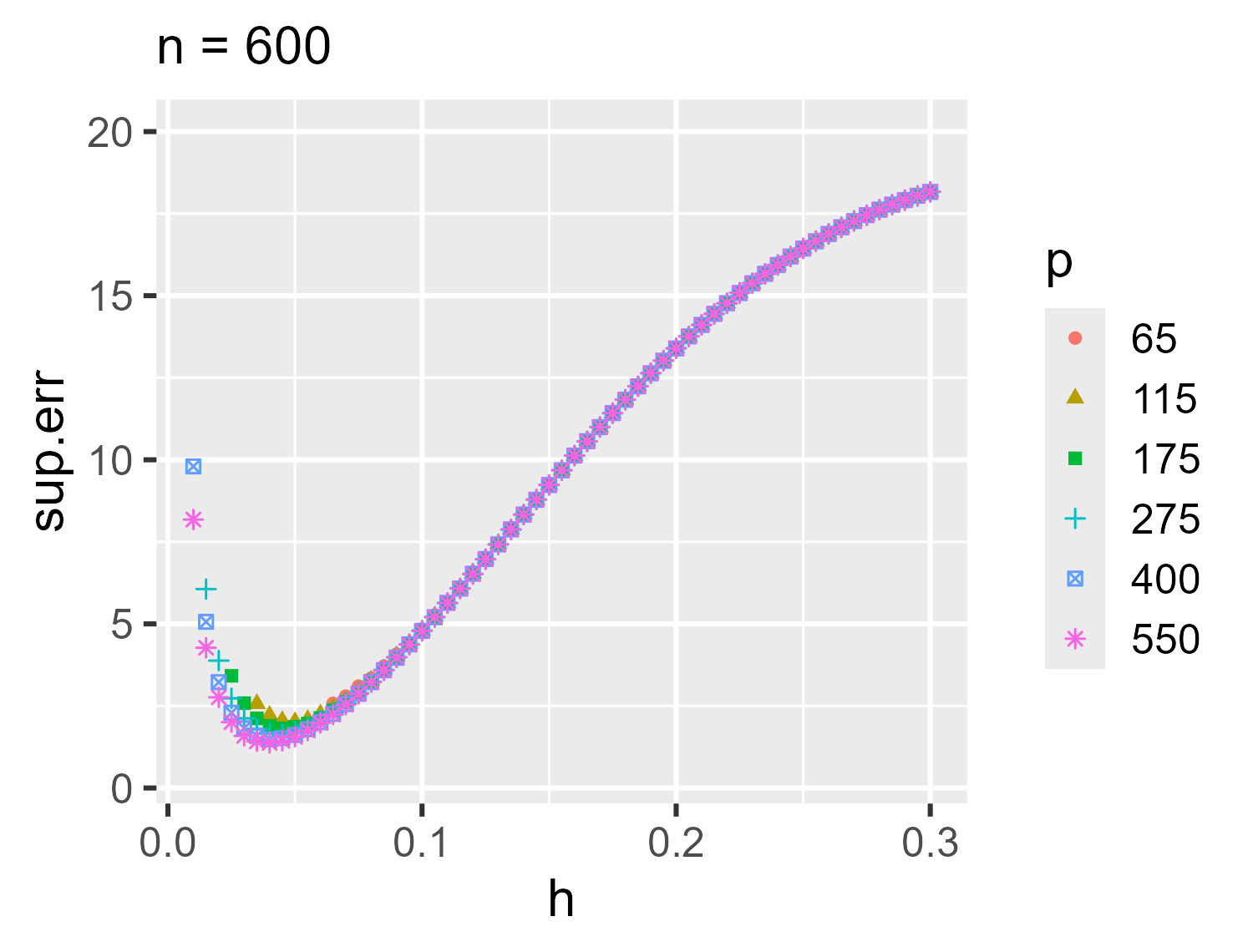}
		\caption{Local quadratic estimator.}
		\label{fig:deriv_bw_comp_quad}
	\end{subfigure}
	\hfill
	\begin{subfigure}{0.49\linewidth}
		\includegraphics[width=\linewidth]{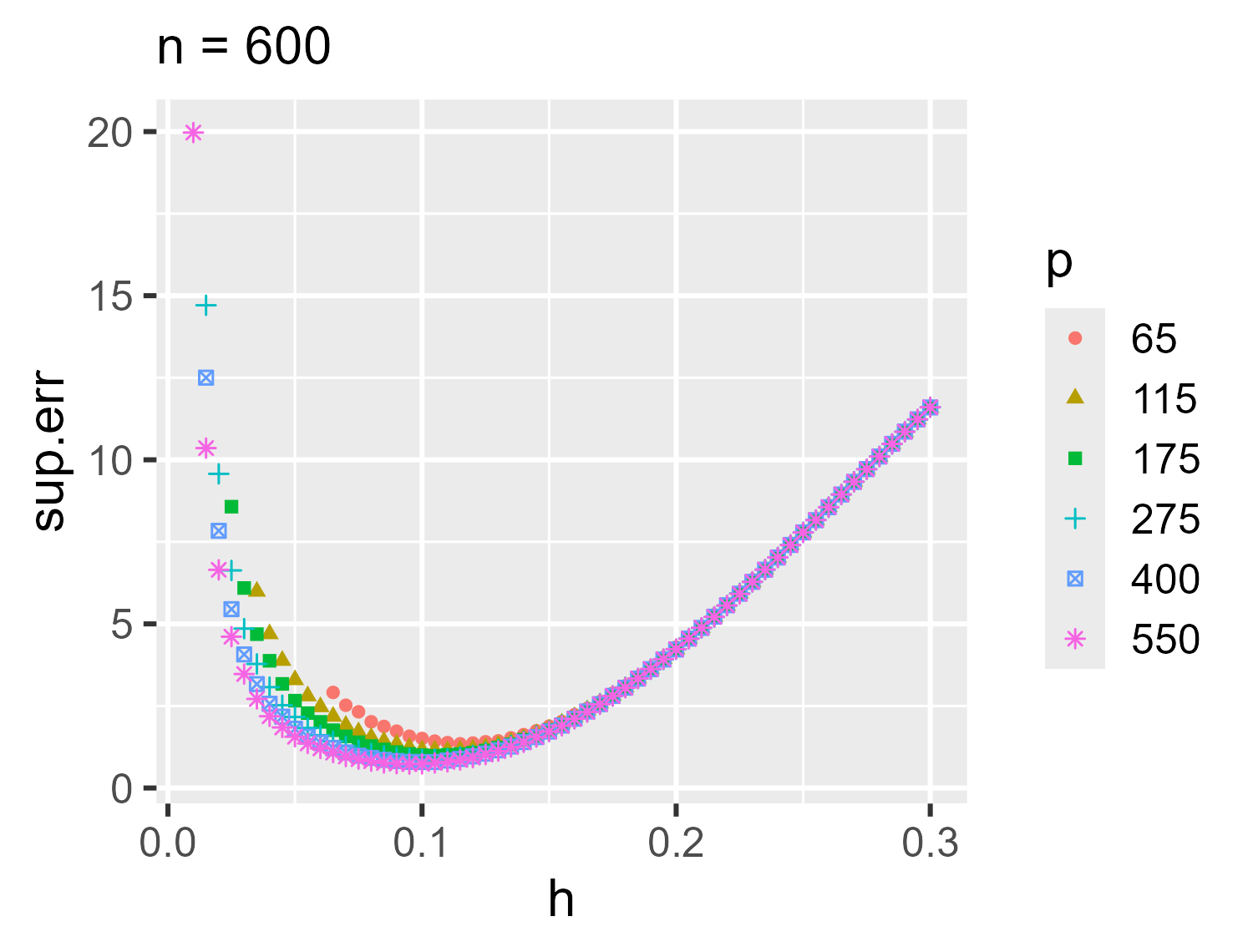}
		\caption{Local cubic estimator.}
		\label{fig:deriv_bw_comp_cubic}
	\end{subfigure}%
	\caption{Bandwidth comparison for the local quadratic (a) and cubic (b) estimator for different $p$ and $n = 600$.}
	\label{fig:deriv_bw_comp}
\end{figure}

\begin{figure}
	\begin{minipage}[b]{.49\linewidth}
		\centering
		\includegraphics[width=\linewidth]{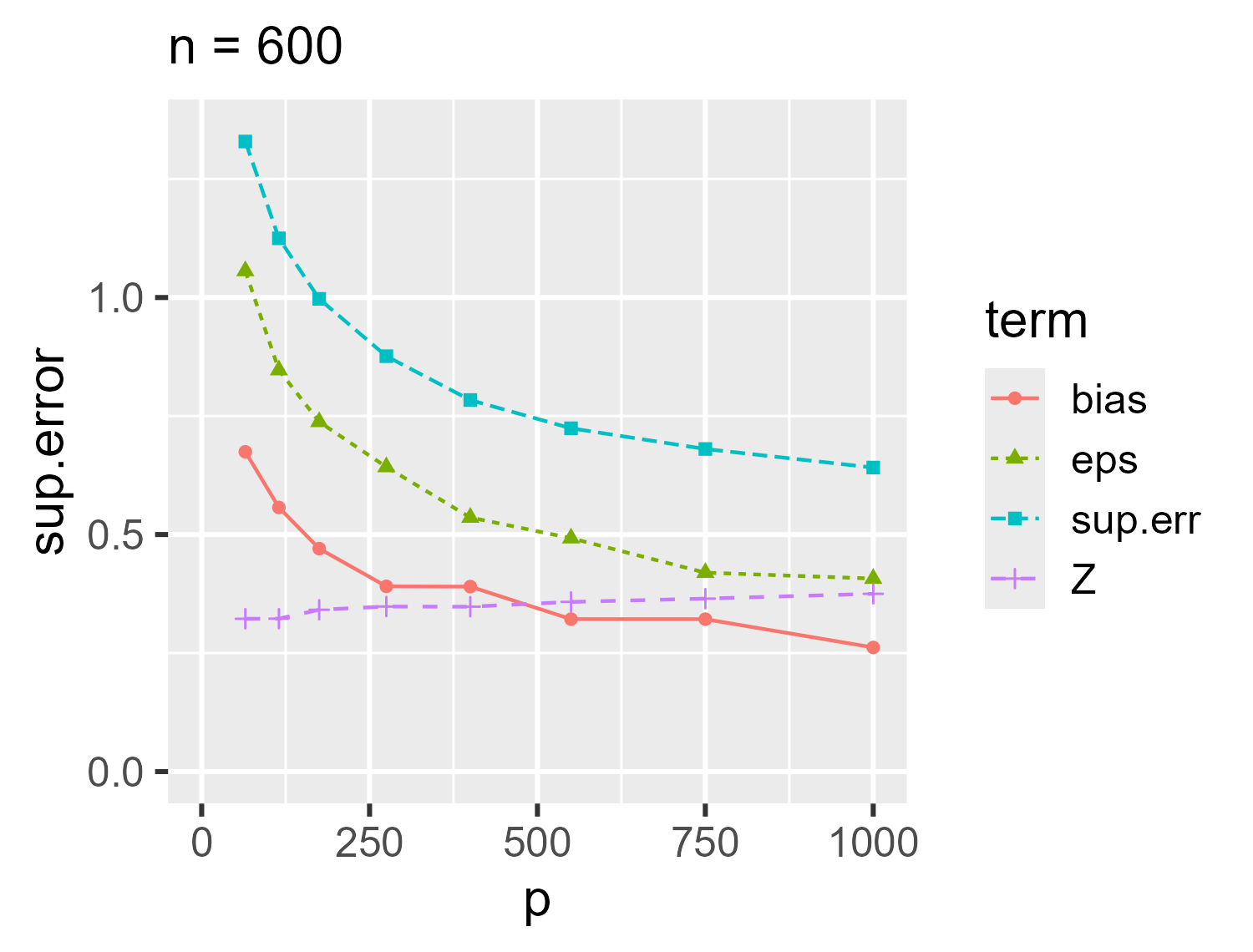}
		\caption{Error decomposition of the local cubic estimator for fixed $n= 600$ and different $p$.}
		\label{fig:error_decomp_n600}
	\end{minipage}
	\hfill
	\begin{minipage}[b]{.49\linewidth}
		\centering
		\includegraphics[width=\linewidth]{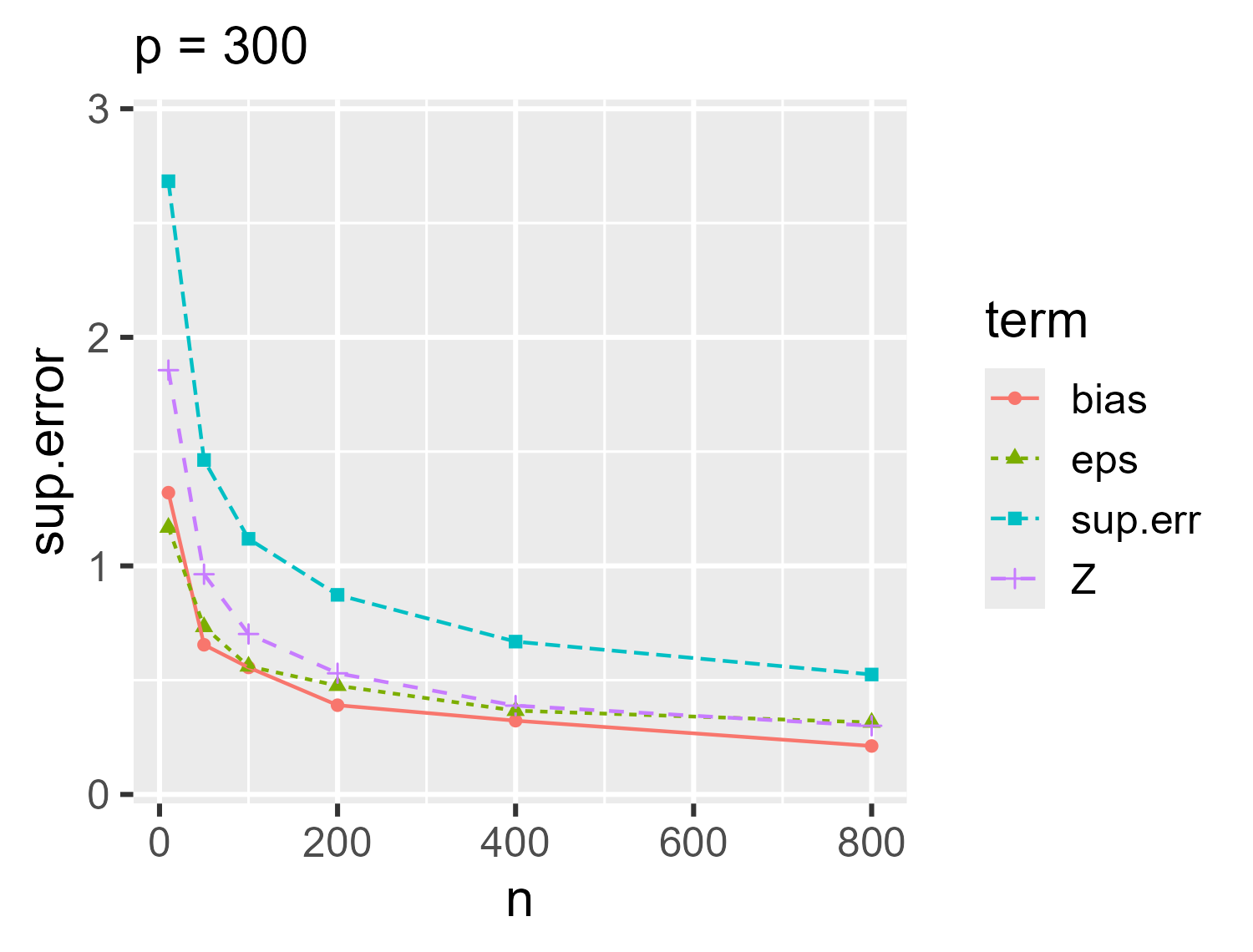}
		\caption{Error decomposition of the local cubic estimator for fixed $p = 300$ and different $n$.}
		\label{fig:error_decomp_p100}
	\end{minipage}
\end{figure}

Finally let us investigate in more detail the influence of smooth and rough stochastic processes in finite samples. 
In addition to the Brownian motion, which has rough paths which are Hölder continuous of  order $\beta < 1/2$, we also consider the smooth processes given by 
\begin{align}
	\tilde Z(x) & \defeq \frac23 \,N_1 \,\sin(\pi\,x) + \frac{\sqrt 8}{3}\,N_2\,\cos(\pi\,x)\,,
\end{align}
where $N_1$ and $N_2$ are independent standard normal distributed random variables. 
For the simulation we use $p = 800$ and $n \in \{10, 20, 40, 80, 160, 240, 480, 800, 1400\}$, and compute  $N = 1000$ realisations of the error terms associated to the stochastic processes, 
$$ \mathrm{I}^{\mathrm{rg}}_n(x) = \frac{1}{n} \sum_{j=1}^p w_{j}^{(1)}(x;h_n) \sum_{i=1}^n Z_i(x_j) \quad \text{and} \quad \mathrm{I}^{\mathrm{sm}}_n(x)\frac{1}{n} \sum_{j=1}^p w_{j}^{(1)}(x;h_n) \sum_{i=1}^n \tilde Z_i(x_j). $$
Our theory, in particular Lemma \ref{lem:mean:estimation:rates:new} implies that $\| \mathrm{I}^{\mathrm{rg}}_n\|_\infty$ should have a rate slower than $1/\sqrt{n\,h_n}$, while $\| \mathrm{I}^{\mathrm{sm}}_n\|_\infty$ should  converge at the $1/\sqrt n$ rate. Table \ref{table:sup_rates_smooth_rough_processes} contains the results, where we used the optimal bandwitdhs for estimation as determined by a grid search. The effect of the smoothness of the paths can be seen in finite samples even though it is not that large (see Figure \ref{fig:error_decomp_n600}).  


\begin{table}
	\centering
	\begin{tabular}{c|ccccccccc}
		$n$ & $10$ & $20$ & $40$ & $ 80$ & $160$ & $240$ & $480$ & $800$ & $1600$ \\
		\hline 
		$\sqrt{nh}\, \mathrm{I}^{\mathrm{rg}}_n$ & $2.690$ & $2.711$ & $2.665$ & $2.657$ & $2.658$ & $2.643$ & $2.644$ & $2.581$ & $2.591$\\
		$\sqrt{n} \,\mathrm{I}^{\mathrm{sm}}_n$ & $3.431$ & $3.595$ & $3.418$ & $3.562$ & $3.466$ & $3.490$ & $3.592$ & $3.690$ & $3.678$\\
        \hline
        $h_n$ &$0.34$ & $0.31$ & $0.28$ & $0.25$ & $0.22$ & $0.19$ & $0.16$ & $0.13$ & $0.1$
	\end{tabular}
	\caption{Supremum norm of the error due to the stochastic processes in the estimation of $\mu_1^\prime$ for smooth and rough processes multiplied with the expected rates by the results of Theorem \ref{thm:mean:upper_bounds}.}
	\label{table:sup_rates_smooth_rough_processes}
\end{table}

\subsection{Real data applications}
\label{sec:real_data_application}

In this section  we revisit the data examples from \citet{berger2023dense}, and complement the analysis with derivative estimation and an investigation of the smoothness of the paths. A data-driven choice of the bandwidth taylored to derivative estimation is rather involved. Therefore we resort to the bandwidth chosen by cross-validation for the mean function. Alternative methods are described e.g.~in \citet{fan1996study}.



\subsubsection*{Application 1: Biomechanics}

We revisit a data set from \citet{liebl2019fast} which contains $n=18$ pairs of torque curves for the right ankle joint in Newton metres (N/m) standardized by the bodyweight (kg) for runners who wear first, extra cushioned running shoes and second and in a second run normal cushioned running shoes. The time of the stance phase (time in which the foot has ground contact) scale is normalized to $t \in [0, 100]$ using affine linear transformations with $p=201$ grid points. See \cite{liebl2019fast} for further details. 


Figure \ref{fig:pressure_curves} shows observations for each of the runners. Pairwise differences together with the averaged observations and a local polynomial estimate of the mean curve are given in Figure \ref{fig:diff_curves_bio}. The bandwidth for the local polynomial estimator is selected by leave-one-curve-out cross validation, resulting in $h = 5.5$. 


\begin{figure}
		\begin{minipage}[b]{.45\linewidth}
		\centering
	\includegraphics[width=\linewidth]{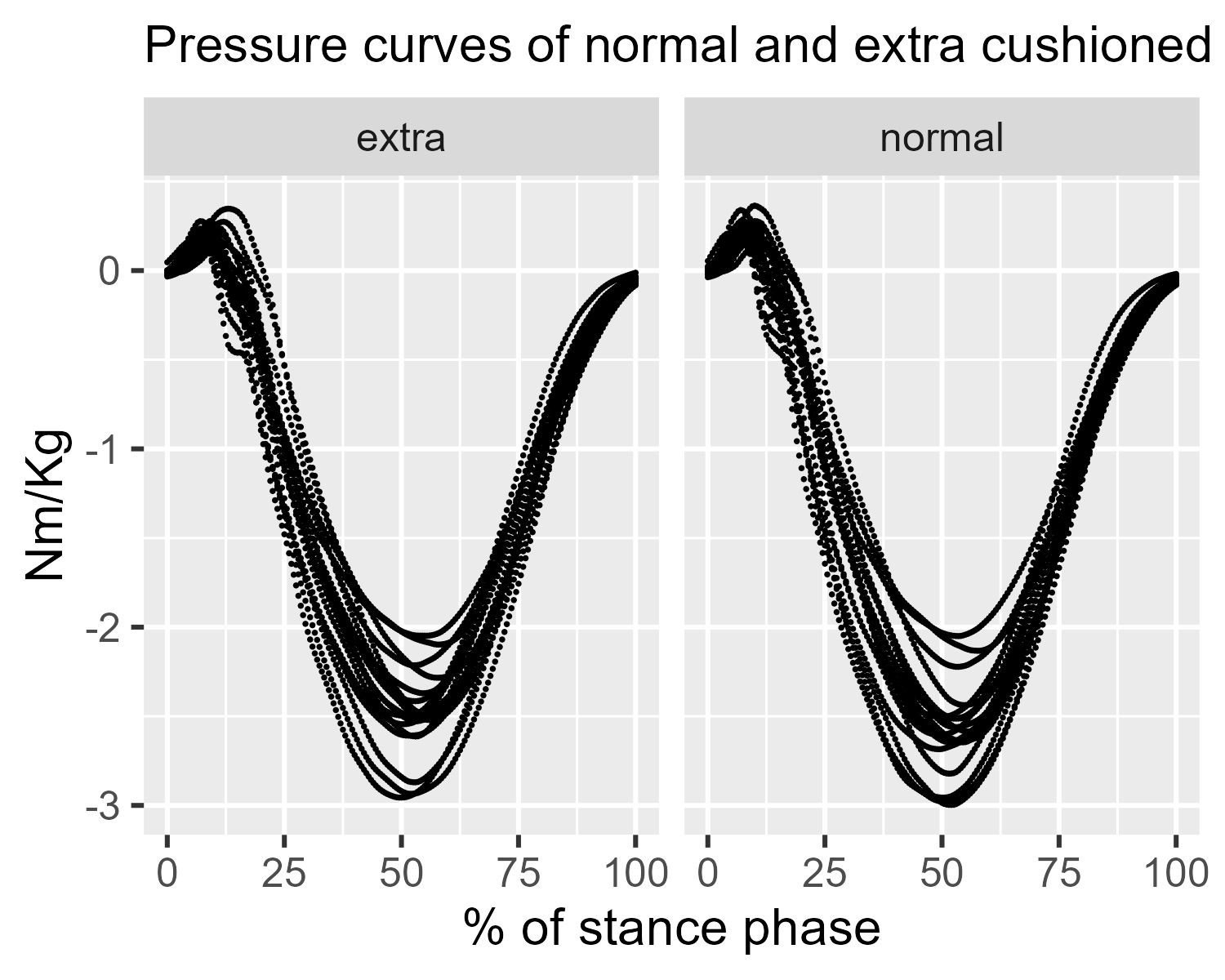}
	\caption{Pressure curves for extra cushioned (left) and the normal (right) shoe.}
	\label{fig:pressure_curves}
	\end{minipage}
	\hfill
	\begin{minipage}[b]{.45\linewidth}
		\centering
		\includegraphics[width=\linewidth]{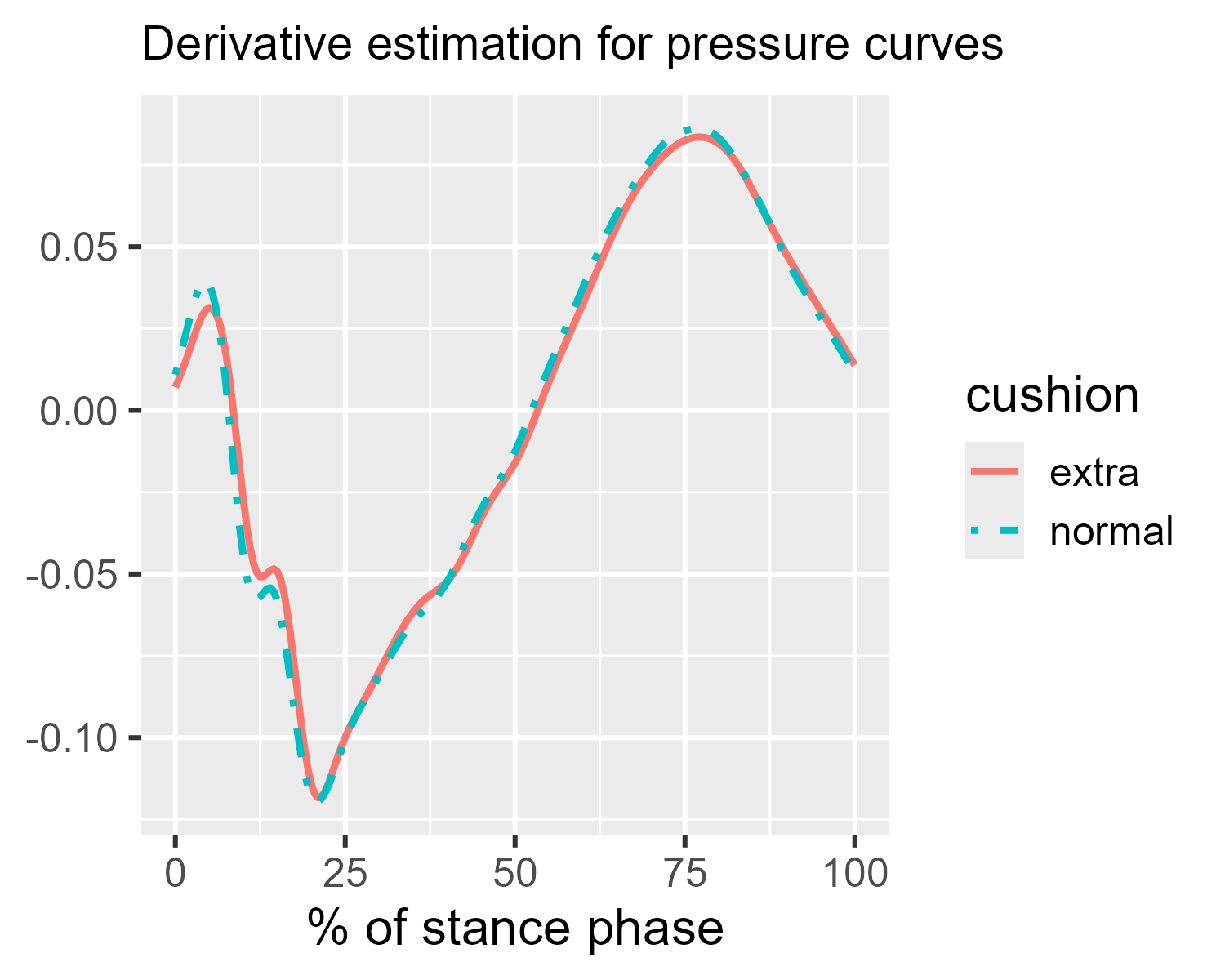}
		\caption{Derivative estimation for the averaged pressure curves per shoe type.}
		\label{fig:bio_deriv_est}
	\end{minipage}
    
\end{figure}

%

\begin{figure}
	\begin{minipage}[b]{.45\linewidth}
		\centering
		\includegraphics[width=\linewidth]{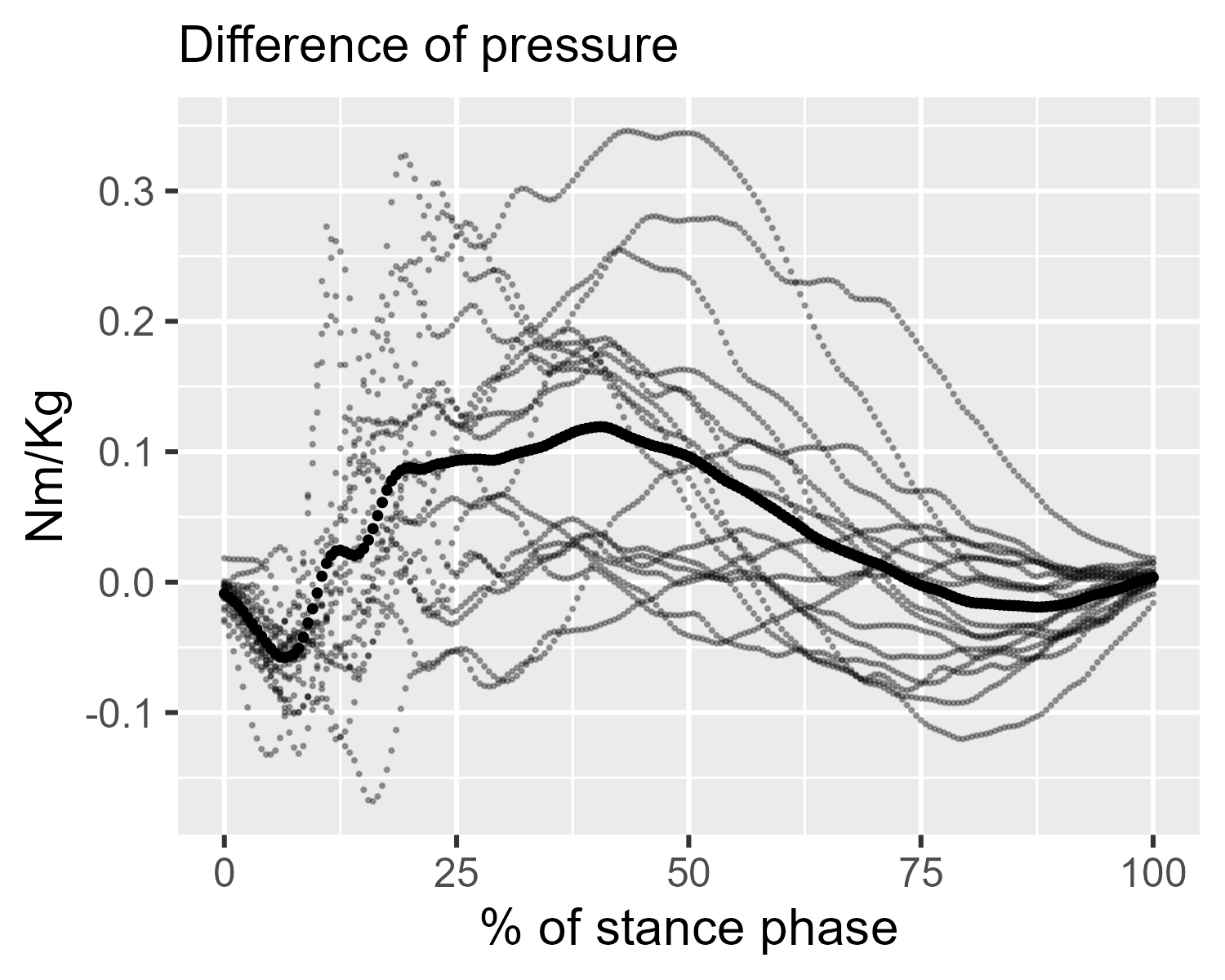}
		\caption{\small Curves of pairwise differences (extra - normal) with  mean function estimate.}
		\label{fig:diff_curves_bio}
	\end{minipage}
    	\hspace{0.5cm}
	\begin{minipage}[b]{.45\linewidth}
		\centering
		\includegraphics[width=\linewidth]{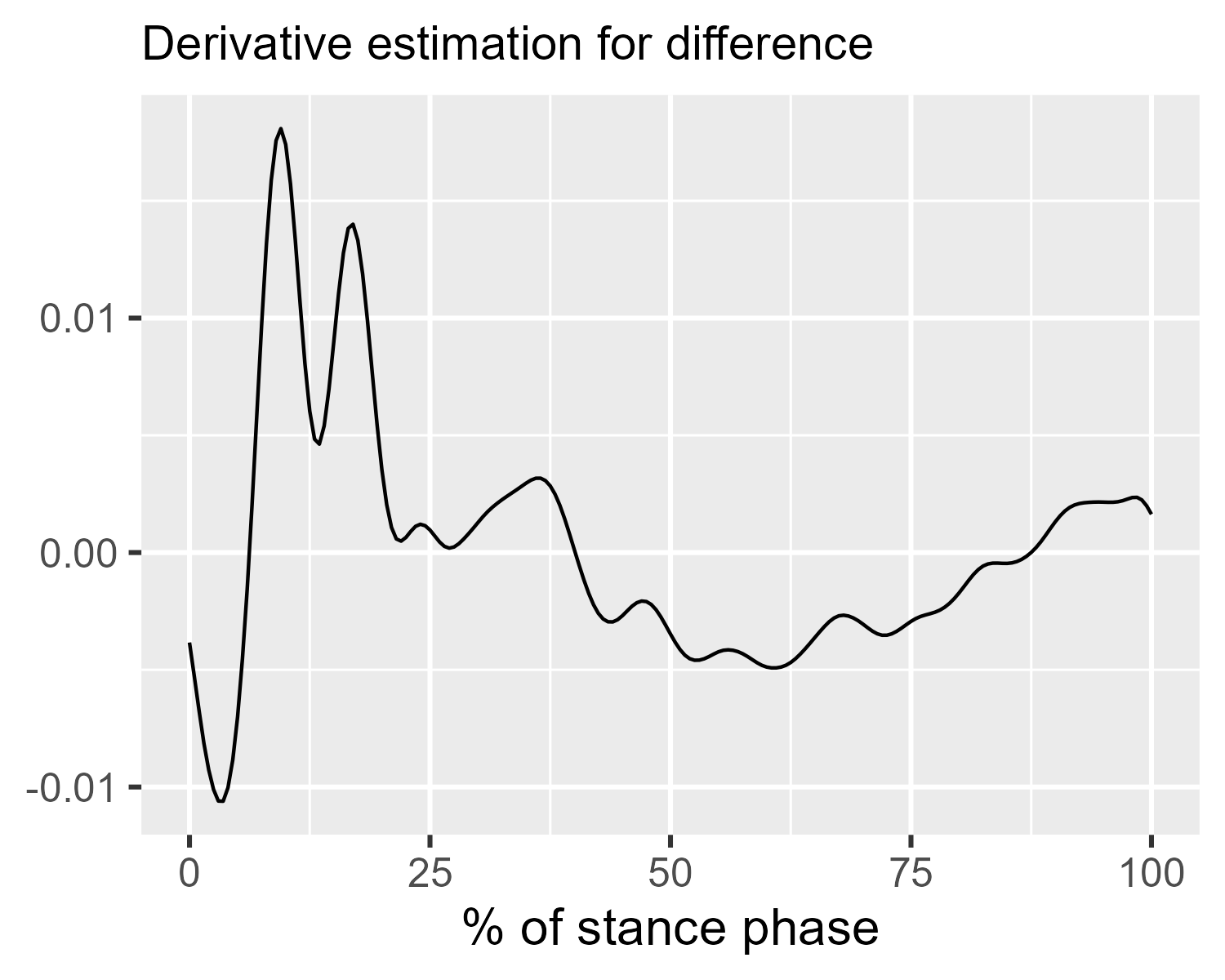}
		\caption{Derivative estimation for the pairwise differences in pressure. }
		\label{fig:bio_deriv_est_diff}
	\end{minipage}	
\end{figure}

Estimating the derivatives of the mean of the pressure curves gives an insight into how fast the ankle is loaded with the pressure. The resulting estimates are contained in Figures \ref{fig:bio_deriv_est} and \ref{fig:bio_deriv_est_diff}. We observe that the non-cushioned shoe has a faster increase when the torque is rising ($\approx [0\%, 10\%]$) and a faster decrease right after that ($\approx [10\%, 20\%]$). A phase-plane plot as suggested in \citet{ramsay1998functional} is given in Figure \ref{fig:bio:phase_plane}. Here we see that the cushioned shoe leads to a slight reduction in the maximal absolute $\mathrm{Nm}/\mathrm{kg}$ but the maximal velocity remains about the same. 

Now it would be desirable to construct confidence bands for the estimate of the derivative of the pairwise difference curve, which could be used to confirm the signs of the derivative. 
Such a construction requires the central limit theorem, Theorem \ref{thm:asymptotic:normality}, for which  the sample paths of the processes need to be sufficiently smooth. To assess this we investigate differentiability of the covariance kernel $\Gamma$ on the diagonal by comparing estimates of $ \partial^{(0,1)}_{u}\Gamma(x,x)$ and $\partial^{(1,0)}_{u}\Gamma(x,x)$ as described in the beginning of this section. The estimates are plotted in Figure \ref{fig:bio_gamm01_gamma10_estimation}. They seem to be quite different, pointing to non-differentiability of the covariance kernel and hence of the sample paths.



\begin{figure}
	\begin{minipage}[b]{.45\linewidth}
		\centering
	    \includegraphics[width = \linewidth]{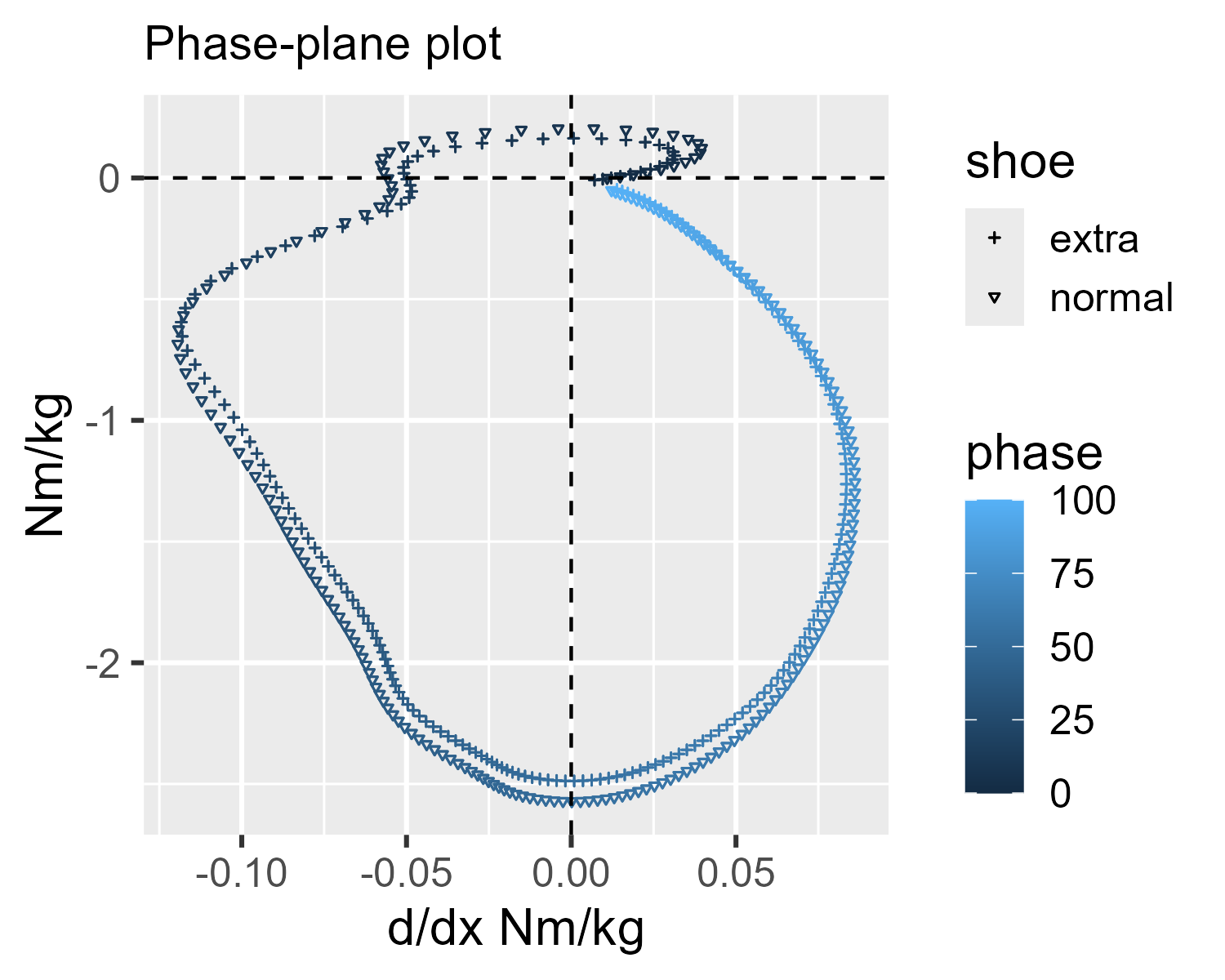}
	   \caption{A \textit{phase-plane plot} of the first derivative against the estimated mean function itself. \\
       \phantom{f} \\
       \phantom{g}}
	\label{fig:bio:phase_plane}
	\end{minipage}
    	\hspace{0.5cm}
	\begin{minipage}[b]{.45\linewidth}
		    \centering
    \includegraphics[width=\linewidth]{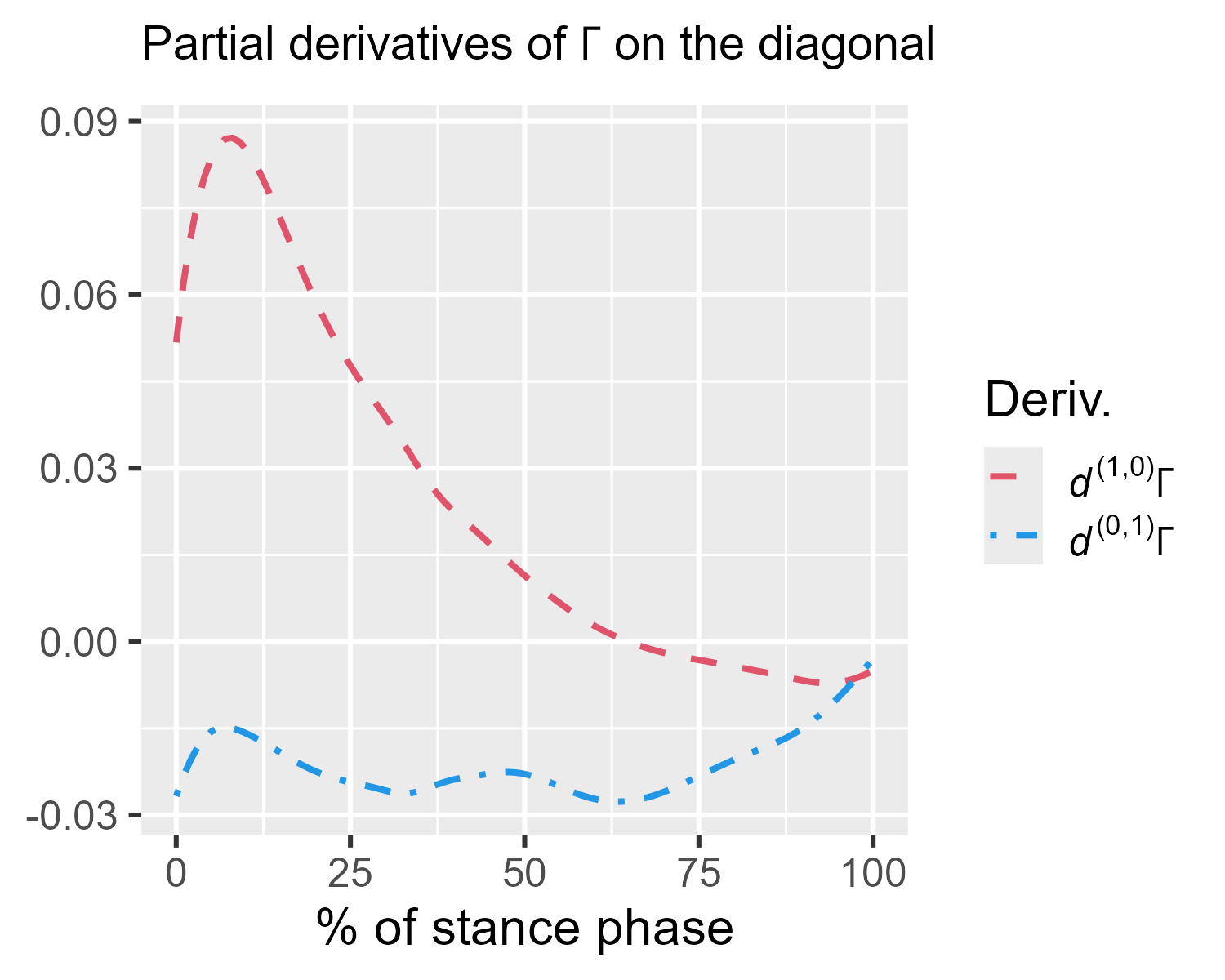}
    \caption{Estimates of $ \partial^{(0,1)}_{u}\Gamma(x,x)$ and $\partial^{(1,0)}_{u}\Gamma(x,x)$ for the  pairwise difference pressure curves of the runners with different shoes, based on the method in  \citet{berger2024optimal}.}
    \label{fig:bio_gamm01_gamma10_estimation}
	\end{minipage}	
\end{figure}


\subsubsection*{Application 2: Temperature data in Nuremberg}

Next we consider daily temperatures (per month) for the years 2000 until 2022 in Nuremberg, Germany. The data were  obtained from the \textit{Deutschen Wetter Dienst (DWD)} at \href{https://opendata.dwd.de/climate_environment/CDC/observations_germany/climate/10_minutes/air_temperature/historical/}{[Link]}. Data are recorded every 10 minutes, which results in $p = 144$. 

Figure \ref{fig:mean_temperature} contains plots of local polynomial fits of the observed daily temperature curves for each month. These are complemented by estimates of the first derivatives in Figure \ref{fig:mean_temperature_derivatives}. Here, in order to alleviate the boundary effect of local polynomial derivative estimation and to account for the periodic nature of temperature over a day, for estimation after / before midnight we additionally use temperature data from the previous / next day.  Summer and spring months together with September have much higher temperature changes over the day than winter months and November, with December being the month with the lowest daily variability.   

Again we assess the smoothness of the paths by comparing estimates $ \partial^{(0,1)}_{u}\Gamma(x,x)$ and $\partial^{(1,0)}_{u}\Gamma(x,x)$ of the covariance kernel $\Gamma$. The plots in Figure \ref{fig:weather_cov_diag} for the month April suggest that these do not coincide, indicating that the paths are actually not smooth. The conclusion can be drawn from Figure \ref{fig:weather_cov_diag_all_months} in the Appendix \ref{sec:simfullint} for all months during the year.  

\begin{figure}
	\begin{minipage}[b]{.45\linewidth}
		\centering
		\includegraphics[width=\linewidth]{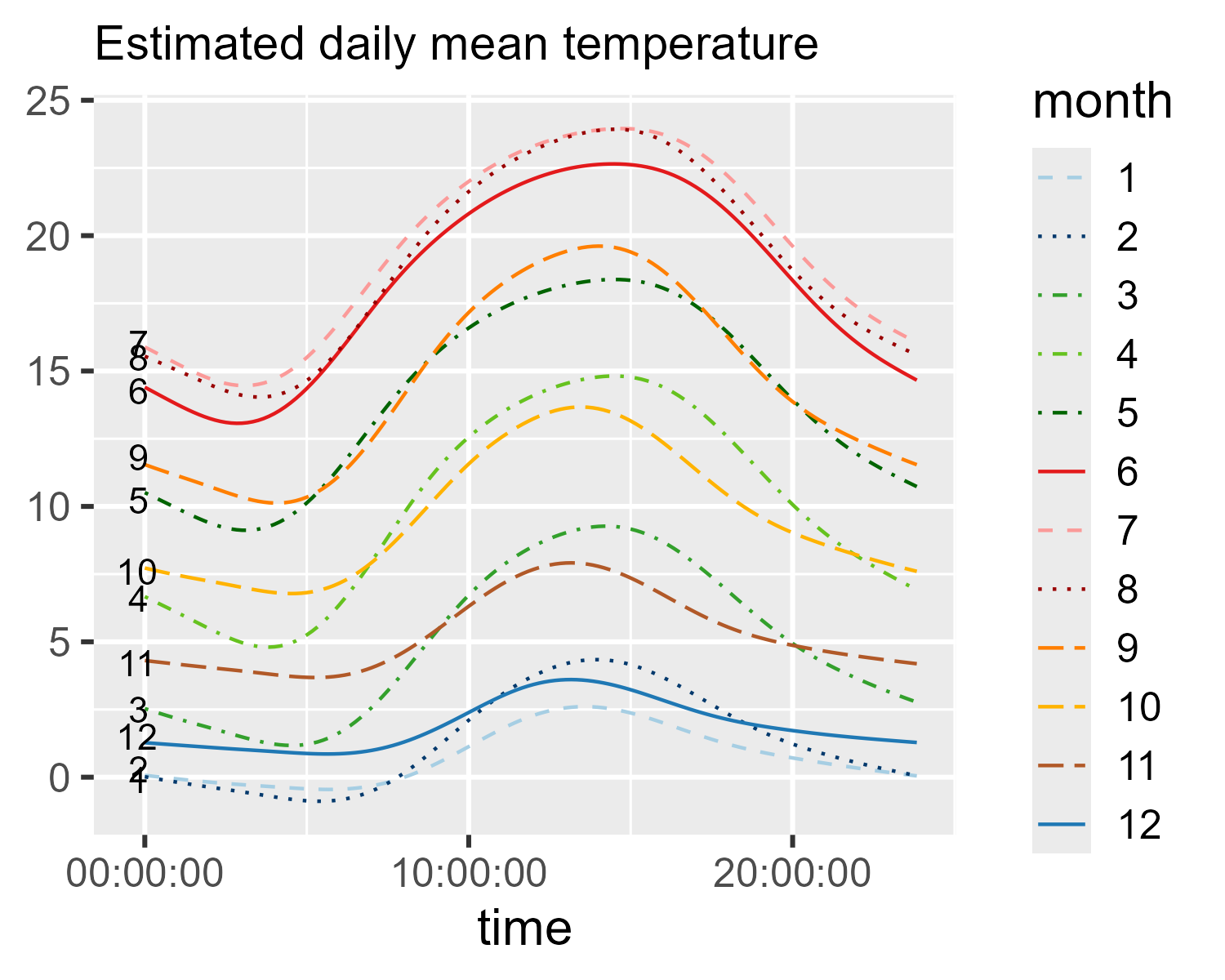}
		\caption{\small Estimation of the mean temperature curves.}
		\label{fig:mean_temperature}
	\end{minipage}
	\hspace{0.5cm}
	\begin{minipage}[b]{.45\linewidth}
		\centering
		\includegraphics[width=\linewidth]{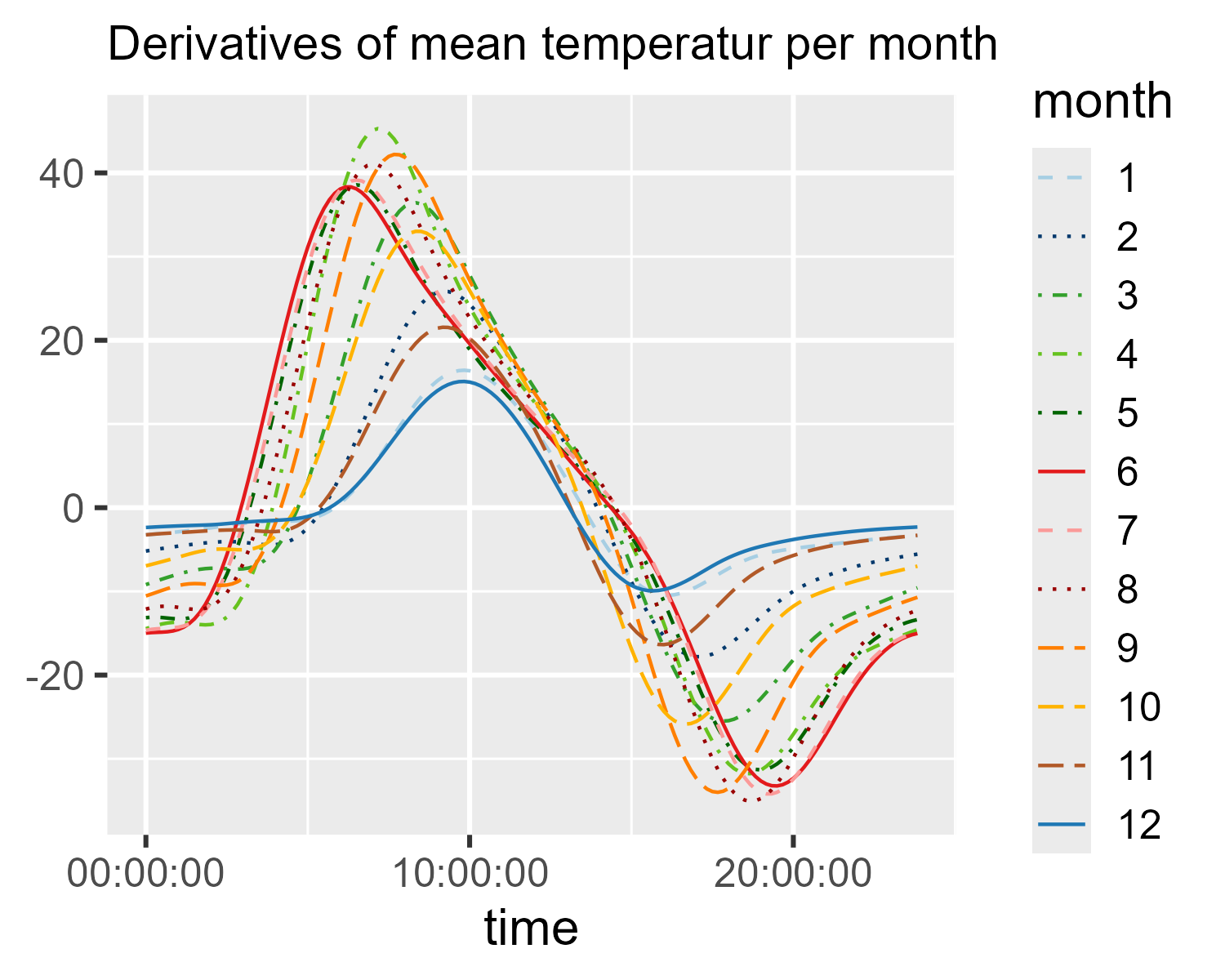}
		\caption{\small Estimation of the first derivative of the average temperature.}
		\label{fig:mean_temperature_derivatives}
	\end{minipage}	
\end{figure}


\begin{figure}
    \centering
    \includegraphics[width=0.5\linewidth]{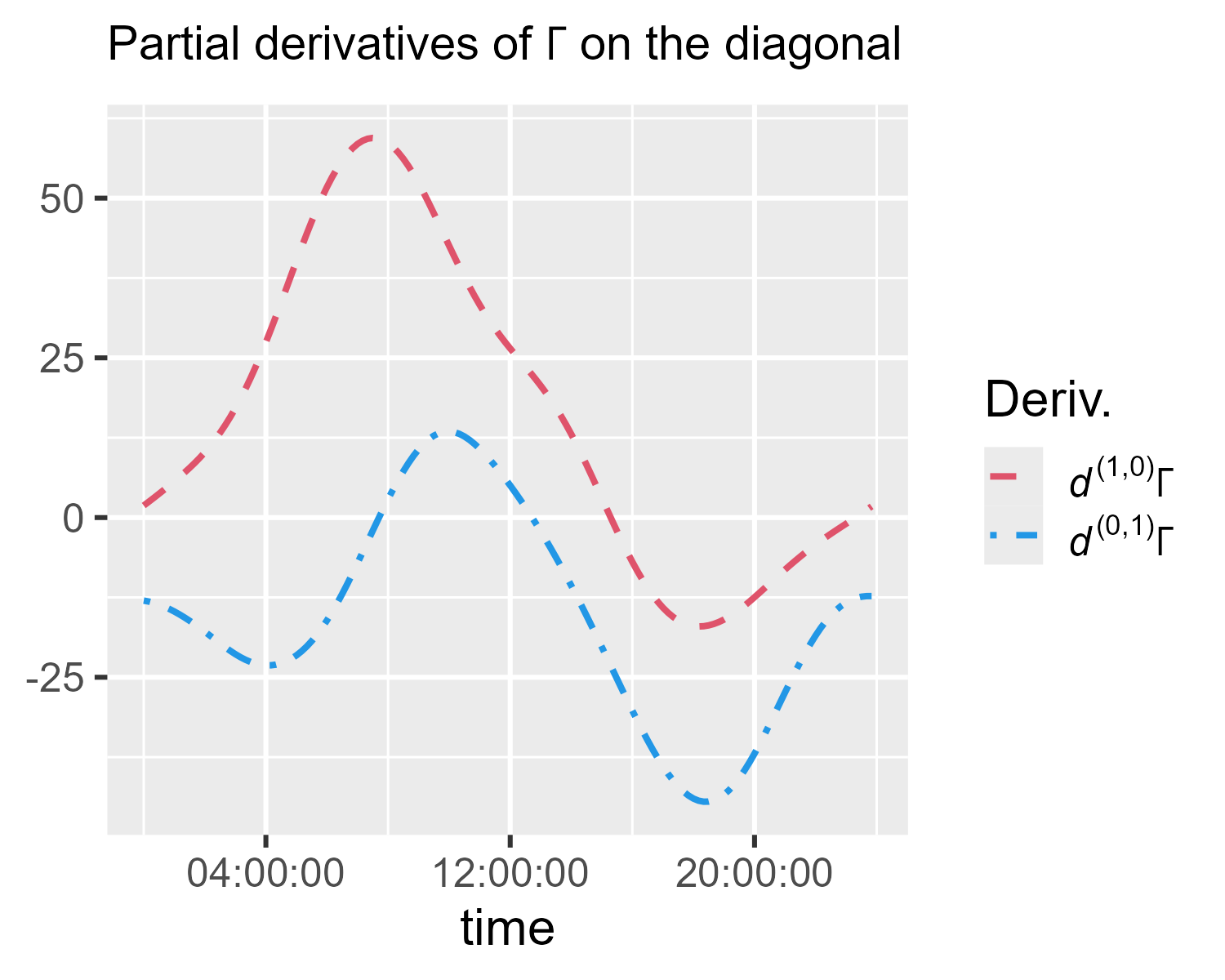}
    \caption{Estimates $ \partial^{(0,1)}_{u}\Gamma(x,x)$ and $\partial^{(1,0)}_{u}\Gamma(x,x)$ of the covariance kernel $\Gamma$ for temperatures in April.}
    \label{fig:weather_cov_diag}
\end{figure}

\section{Concluding remarks}\label{sec:conclude}

In this paper we discussed derivative estimation for the mean function of synchronously observed functional data. Estimating the derivatives of the covariance kernel and of the principal component functions is also of quite some interest \citep{dai2018derivative}. Smoothness of the covariance kernel near the diagonal is intrinsically related to smoothness of paths of the associated process \citep{azmoodeh2014necessary, da2023sample}, thus the phenomenon of rough sample paths with a smooth mean function will not analogously arise for the covariance kernel. As we indicated in Section \ref{sec:real_data_application}, using the restricted local polynomial estimator from \citet{berger2024optimal}, we can investigate differentiability of the covariance kernel on the diagonal, which is a necessary condition for smoothness of the sample paths of the associated processes. In subsequent work we plan to make use of this for the construction of a formal statistical testing procedure.

\section{Proofs for Section \ref{sec:optimalrates}}\label{sec:proof:theorem:rates}

\subsection{Proof of Theorem \ref{thm:mean:upper_bounds}}
Similar to the analysis for the estimation error of the mean function $\mu$ itself in \citet{berger2023dense}, if the weights $\wksb xhs$ fulfil \ref{ass:weights:pol:rep} then the error can be decomposed into
\begin{align}
	\hat \mu_{n,\bs p,h}^{(\bs s)}(\bs x) - \dels \mu(\bs x) & =  \sum_{\bs k=\bs 1}^{\bs p} \wksb xhs\,\big(\mu(\bs x_{\bs k}) - \frac{(\bs{x_j} - \bs x)^{\bs s}}{\bs s!}\dels \mu(\bs x)\big ) + \sum_{\bs k=\bs 1}^{\bs p} \wksb xhs\,\bar{\epsilon}_{\bs k} \\ & \qquad + \sum_{\bs k=\bs 1}^{\bs p} \wksb xhs\, \bar{Z}_{n}(\bs x_{\bs k})\nonumber\\
	& \defeql  I_{1}^{\bs p,h, \bs s}(\bs x) + I_{2}^{n, \bs p,h, \bs s}(\bs x) + I_{3}^{n, \bs p,h, \bs s}(\bs x) \, , \qquad \bs x \in[0,1]^d \,, \label{eq:decomposition}
\end{align}	
where as above
\begin{align*}
	\bar \epsilon_{\bs j,n} = \frac1n \sum_{i=1}^n \varepsilon_{i,\bs j} \qquad \text{and} \qquad \bar{Z}_{n}(\bs x) = \frac1n \sum_{i=1}^n Z_i(\bs x) \, .
\end{align*}

From the error decomposition \eqref{eq:decomposition} we obtain the bound
\begin{align}
	& \, \sup_{\mu\in\hclass, \, Z \in \Pclass} \E_\mu \Big[ \normb{ \muestsb{x}{h} - \dels \mu}_\infty \Big]\nonumber\\
	\leq & \, \sup_{\mu\in\hclass} \normb{I_1^{\bs p,h, \bs s}}_\infty
	+ \E  \Big[\normb{I_2^{n, \bs p,h, \bs s}}_\infty  \Big]+\sup_{Z \in \Pclass} \E  \Big[ \normb{I_3^{n, \bs p,h, \bs s}}_\infty  \Big] \,.\label{eq:errordecomp}
\end{align}

Let us note that property \ref{ass:weights:max} in Assumption \ref{ass:weights} together with Assumption \ref{ass:design:distribution} implies the following property. 
\begin{enumerate}[label=\normalfont{(M\arabic*)},leftmargin=9.9mm]
	\setcounter{enumi}{4}
	\item 	We have  $ \displaystyle \sum_{\bs 1\leq \bs j \leq \bs p} \big|\wjsb xhs \big| \leq \Csum\,h^{-\abs{\bs s}}$, $\bs x \in T$ with $\Csum=\Cmax\Ccard$. \label{ass:weights:sum}
\end{enumerate}

The upper bounds in Theorem \ref{thm:mean:upper_bounds}  directly follow from Lemmas \ref{lem:mean:estimation:rates:asbefore} and \ref{lem:mean:estimation:rates:new} below. 

\begin{lemma} \label{lem:mean:estimation:rates:asbefore} 
For the terms $I_1^{\bs p,h, \bs s}$ and $I_2^{n, \bs p,h, \bs s}$ in the error decomposition given in \eqref{eq:decomposition}, we have the following bounds.
\begin{enumerate}
	\item If the weights $\wjsb xhs$ satisfy \ref{ass:weights:pol:rep} with $\zeta = \lfloor\alpha\rfloor$ of Assumption \ref{ass:weights} as well as \ref{ass:weights:sum}, then 
	$$ \sup_{\mu \in \mc H(\alpha, L)} \sup_{h \in (\bs c/ \bs p, h_0]} h^{-(\alpha - \abs{\bs s})} \normb{I_1^{n,\bs p,\bs s}}_\infty = \mc O \big( L\big)\,. $$
	\item If the weights $\wjsb xhs$ satisfy \ref{ass:weights:equal:0} and \ref{ass:weights:max} of Assumption \ref{ass:weights}, the design satisfies Assumption \ref{ass:design:distribution} and the errors Assumption \ref{ass:model}, then uniformly in $h \in (\bs c/ \bs p, h_0]$, 
	$$ \expec\Big[ \normb{I_2^{n, \bs p, h,\bs s}}_\infty\Big] = \mc O \bigg( \sqrt{\frac{\log(h^{-1})}{n\,\bs{p^1}\,h^{d  + 2\abs{\bs s}}}} \bigg)\,. $$
\end{enumerate}
\end{lemma}

\begin{proof}[Proof of Lemma \ref{lem:mean:estimation:rates:asbefore}]

The proof of part \textit{ii)} is fully analogous to that for the mean function $\mu$ itself as given in \citet[Lemma 2, ii)]{berger2023dense}. For convenience it is given in the supplementary Appendix, Section \ref{app:additional_proofs}. 

\textit{i).}  While the proof of part \textit{i)} is also similar to the case of the mean function, some care has to be taken in the multivariate setting since we only estimate one of the partial derivatives of order $\abs{ \bs s}$. 

From the multivariate Taylor theorem of order $\floor{\alpha}$ with Lagrange remainder term, given $\bs x$ for each $\bs x_j$ there exists  $\theta_{\bs j} \in [0,1], \bs 1 \leq \bs j \leq \bs p,$ such that for $\bs{\tau_j} \defeq \bs x + \theta_{\bs j}(\bs{x_j} - \bs x) \in [0,1]^d$  it holds that 
$$ \mu(\bs x_{\bs j})  = \sum_{|\bs k|= \bs 0}^{\floor{\alpha}-1}\partial^{\bs k} \mu (\bs x) \frac{(\bs x_{\bs j}-\bs x)^{\bs k}}{\bs k !} +  \sum_{|\bs k|= \floor{\alpha}}\partial^{\bs k} \mu (\bs \tau_{\bs j})\frac{(\bs x_{\bs j}-\bs x)^{\bs k}}{\bs k !}. $$
Therefore
	\allowdisplaybreaks
\begin{align*}
	\Big|\sum_{\bs j = \bs 1}^{\bs p}& \wjsb xhs \mu(\bs x_{\bs j}) - \dels\mu(\bs x)\Big|\\
	& = \Big|\sum_{\bs j = \bs 1}^{\bs p} \wjsb xhs\big( \mu(\bs x_{\bs j}) - \frac{(\bs{x_j} - \bs x)^{\bs{s}}}{\bs s!}\dels\mu(\bs x)\big)\Big| \tag{by \ref{ass:weights:pol:rep}}\\
	&=\Big|\sum_{\bs j = \bs 1}^{\bs p} \Big(\sum_{|\bs k|= \bs 0}^{\floor{\alpha}-1}\partial^{\bs k} \mu (\bs x) \frac{(\bs x_{\bs j}-\bs x)^{\bs k}}{\bs k !} +  \sum_{|\bs k|= \floor{\alpha}}\partial^{\bs k} \mu (\bs \tau_{\bs j})\frac{(\bs x_{\bs j}-\bs x)^{\bs k}}{\bs k !}\\ 
	& \QQuad\QQuad -  \frac{(\bs{x_j} - \bs x)^{\bs{s}}}{\bs s!}\dels \mu (\bs x)\Big) \wjsb xhs\Big|\tag{Taylor}\\
	&=\Big|\sum_{\bs j = \bs 1}^{\bs p}\Bigg( \sum_{|\bs k|= \floor{\alpha}}\bigg(\partial^{\bs k} \mu (\bs \tau_{\bs j})-\partial^{\bs k} \mu (\bs x)\bigg)\frac{(\bs x_{\bs j}-\bs x)^{\bs k}}{\bs k !}\Bigg) \wjsb xhs\Big|\tag{by \ref{ass:weights:pol:rep}}\\
	& \leq L\, \sum_{\bs j = \bs 1}^{\bs p}\Bigg( \sum_{|\bs k|= \floor{\alpha}}\norm{\theta_{\bs j}(\bs{x_j} - \bs x)}_\infty^{\alpha-\floor{\alpha}}\frac{\big|(\bs x_{\bs j}-\bs x)^{\bs k}\big|}{\bs k !}\Bigg) \big|\wjsb xhs\big|\\
	&\leq  L\, \sum_{\bs j = \bs 1}^{\bs p}\norm{\bs x_{\bs j}- \bs x}_\infty^{\alpha}\Bigg( \sum_{|\bs k|= \floor{\alpha}}\frac{1}{\bs k !}\Bigg) \big|\wjsb xhs\big|\\
	&\leq L\, h_\infty^{\alpha  - \abs{\bs s} } \sum_{|\bs k|= \floor{\alpha}}\frac{C_1}{\bs k !}.
\end{align*}

\end{proof}

\begin{lemma} \label{lem:mean:estimation:rates:new} 
For the term $I_3^{n, \bs p,h, \bs s}$ in the error decomposition given in \eqref{eq:decomposition}, if the weights $\wjsb xhs$ satisfy \ref{ass:weights:equal:0}, \ref{ass:weights:max} and \ref{ass:weights:lipschitz} we have the following bounds.
\begin{enumerate}
	\item If $\beta > \abs{\bs s}$ then for sufficiently large $n$, 
   	\begin{align}
		 \sup_{h \in (\bs c/ \bs p, h_0]}\sup_{ Z  \in \mc P(\beta, C_Z)} \expec\Big[ \normb{ \, I_3^{n,\bs p,h,(\bs s)}}_\infty \Big] \leq C \, \frac{C_Z}{\sqrt n} , \label{eq:bound:smooth:process}
	\end{align}
for some universal constant $C>0$. 
    \item If $\beta \leq \abs{\bs s}$ then for sufficiently large $n$, 
    \begin{align}
		 \sup_{h \in (\bs c/ \bs p, h_0]}\sup_{ Z  \in \mc P(\beta, C_Z)} \,\sup_{0<\delta <\min(1,\beta) } \delta^{1/2}\,h^{\abs{\bs s} - (\beta - \delta) }\,\expec\Big[ \normb{ I_3^{n,\bs p,h,\bs s}}_\infty \Big] \leq C\, \frac{C_Z}{\sqrt n}. \label{eq:bound:rougher:paths:then:derivative}
	\end{align}
	
\end{enumerate}
\end{lemma}

\begin{proof}[Proof of Lemma \ref{lem:mean:estimation:rates:new}]

We shall apply the maximal inequality from \citet[Section 7, display (7.10)]{pollard1990empirical}, see Lemma \ref{lem:maximalInequalityExpection} in Section \ref{sec:maxinequa},  to the processes  $X_{n} = (X_{n1},\ldots,X_{nn})$ defined by
	\begin{align}
		X_{ni} (\bs x) = \frac{1}{\sqrt n} \sum_{\bs j = \bs 1}^{\bs p} \wjsb xhs  Z_i(\bs{x_j}), \label{eq:X:process:deriv}
	\end{align}
where $Z_i \in \mc P(\beta, C_Z)$. To this end let $M_i \geq \norm{Z_i}_{\mc H, \beta}$ \text{ a.s.} with $\expec[M_i^2]  \leq C_Z$ as required in the definition \eqref{eq:classprocesses} of the class $\mc P(\beta, C_Z)$.   

\smallskip

\textit{i)} We derive upper bounds on the capacity number as required for the maximal inequality in Lemma \ref{lem:maximalInequalityExpection}. To this end we check the conditions of Lemma \ref{lem:manageability:lipschitz} where we use the envelope $\bs \Phi_{n} \defeq  \Cenv\, (M_1, \ldots , M_n)^\top / \sqrt n$. Here $\Cenv>0$ is appropriate constant depending on $\Cmax$ - $ \Csum$ from Assumption \ref{ass:weights} on the weights, as well as $\bs s, \alpha, C_Z, \beta$ and $d$. 
Indeed this an envelope for $X_n$, since by using Taylor expansion up to order $\abs{\bs s} $, for intermediate points $\bs {\tau_j} = \bs x + \theta_{\bs j}(\bs{x_j} - \bs x), \theta_{\bs j} \in [0,1]$  and the property \ref{ass:weights:pol:rep} of the weights,   we get that
	\begin{align*}
		\abs{X_{ni}(\bs x)} & = \frac1{\sqrt n}  \Big| \sum_{\bs j = \bs 1}^{\bs p} \wjsb xhs Z_i(\bs{x_j}) \Big|\\
        & = \frac1{\sqrt n} \absb{\sum_{\bs j = \bs 1}^{\bs p} \wjsb xhs \Big( \sum_{\abs{\bs k} = 0}^{ \abs{\bs s}-1} \Zi kx\, \frac{(\bs{x_j} - \bs x)^{\bs k}}{\bs k!} + \sum_{\abs{\bs k} = \abs{\bs s}} \Zi k{ \tau_j}\, \frac{(\bs{x_j} - \bs x)^{\bs k}}{\bs k!}\Big)} \\
		& \leq \frac1{\sqrt n} \, M_i\, \sum_{\bs j = \bs 1}^{\bs p} |\wjsb xhs|\, \sum_{\abs{\bs k} = \abs{\bs s}}  \frac{\big|(\bs{x_j} - \bs x)^{\bs k}\big|}{\bs k!}\Big)\tag{by \ref{ass:weights:pol:rep}}\\
		& \leq \frac{ \Csum}{\sqrt n\,h^{\abs{\bs s}}}\, M_i \, h^{\abs{\bs s}}\, \sum_{\abs{\bs k} = \abs{\bs s}}   \frac{1}{\bs k!} \leq \frac{ \Csum}{\sqrt n}\, M_i \,. \tag{by \ref{ass:weights:sum}}
	\end{align*}

We next check the Lipschitz property as stated in Lemma \ref{lem:manageability:lipschitz}, which is used to bound the capacity number. To this end let $0< \epsilon< 1$, $\bs x ,\bs y$ with 
\begin{equation}\label{eq:lipschitzcont}
\norm{\bs x - \bs y}_\infty < \epsilon^\eta,
\end{equation}
where $\eta>0$ is to be specified. Further let $\gamma = \min(1,\beta - \abs{\bs s})$. 


	First if $\epsilon > h$, we take $\eta = 1/\gamma$. Using Taylor expansion up to order $\abs{\bs s} $ and the property \ref{ass:weights:pol:rep} of the weights, for intermediate points $\bs \tau_{\bs x,\bs j} = \bs x + \theta_{\bs x,\bs j}(\bs{x_j} - \bs x), \theta_{\bs x,\bs j} \in [0,1]$ and $\bs \tau_{\bs y,\bs j} = \bs y + \theta_{\bs y,\bs j}(\bs{x_j} - \bs y), \theta_{\bs y,\bs j} \in [0,1]$ we obtain
	\begin{align}
		\sqrt n \absb{ X_{ni}(\bs x) - X_{ni}(\bs y)}
		& \leq \abss{ \sum_{\bs j = \bs 1}^{\bs p} \wjsb xhs \sum_{\abs{\bs k} = \abs{\bs s}}\frac{(\bs{x_j} - \bs x)^{\bs k}}{\bs k!} \big( \Zi k{\bs \tau_{\bs x,\bs j}} - \Zi kx\big)} \nonumber\\
        & \QQuad + \absb{\Zi sx - \Zi sy} \label{eq:hilfboundprocess}\\
		& \QQuad + \abss{ \sum_{\bs j = \bs 1}^{\bs p} \wjsb yhs \sum_{\abs{\bs k} = \abs{\bs s}} \frac{(\bs {x_j} - \bs y)^{\bs k}}{\bs k!} \big( \Zi ky - \Zi k{\bs \tau_{\bs y,\bs j}}\big)}. \nonumber
	\end{align}
	Now for the second term using the Hölder continuity of $\Zi sx$ yields
	\begin{align*}
		 \absb{\Zi sx - \Zi sy} \leq \Cenv\, M_i\, \norm{\bs x - \bs y}_\infty^\gamma \leq \,\Cenv \, M_i\, \epsilon\,
	\end{align*}
    by \eqref{eq:lipschitzcont} and choice of $\eta=1/\gamma$. 
	Using \ref{ass:weights:sum} the first and third terms are similarly bounded by 
	\begin{align*}
		 \sum_{\bs j = \bs 1}^{\bs p} \abs{\wjsb xhs} \sum_{\abs{\bs k} = \abs{\bs s}} \frac{\abs{(\bs{x_j} - \bs x)^{\bs k}}}{\bs k!} \abs{ \Zi{k}{\bs \tau_{\bs x,\bs j}} - \Zi kx} & \leq \Cenv\, \frac{1}{h^{\abs{\bs s}}}\,h^{\abs{\bs s}}\,M_i\,h^\gamma \leq \Cenv\, \, M_i\, \epsilon^\gamma . 
	\end{align*}
    since we considered the case that $\epsilon > h$.
    
	Now if $\epsilon \leq h$, we argue as for the case \textit{i)}
    if we take $\eta = 2 + (d + \abs{\bs s})$ in \eqref{eq:lipschitzcont}, using \ref{ass:weights:lipschitz} we get 
	\begin{align*}
		\sqrt n \, \absb{X_{n,i}(\bs x) - X_{n,i}(\bs y)} & \leq M_i \frac{\Ccard}{h^{d + \abs{\bs s}}}\bigg(\frac{\norm{\bs x - \bs y}_\infty}{h} \wedge 1 \bigg)\\
		& \leq M_i \frac{\Ccard}{h^{d + \abs{\bs s}}}\, \frac{\epsilon^{2 + (d + \abs{\bs s})}}{h} \tag{choice of $\eta$ in \eqref{eq:lipschitzcont}}\\
		& \leq \Ccard\,M_i\epsilon\,.\tag{$\epsilon \leq h$}
	\end{align*}
	Then the bound in \textit{i)} follows from Lemmas \ref{lem:manageability:lipschitz} and \ref{lem:maximalInequalityExpection}. 
    
	\textit{ii).} We proceed similarly as for \textit{i)}, but using the envelope  
    $$\Phi_{n}=\frac{\Cenv\, (M_1, \ldots, M_n)^\top}{\sqrt n\, h^{\abs{\bs s} -(\beta - \delta)}}.$$
    To show that this is indeed an envelope for $X_n$ for appropriate choice of $\Cenv>0$, again we use a Taylor expansion, this time of order $\floor{\beta}$. Then for intermediate points  
    $\bs{\tau_j} = \bs x +  \theta_{\bs j}(\bs{x_j} - \bs x)$ with certain $\theta_{\bs j}\in[0,1]$ we obtain the bound
	\begin{align*}
		\big| X_{n,i}(\bs x) \big| & = \bigg| \frac1{\sqrt n} \sum_{\bs j =\bs 1}^{\bs p} \wjsb xhs \Big( \sum_{\abs{\bs k} = 0}^{ \floor{\beta}-1} \frac{(\bs{x_j} - \bs x)^{\bs k}}{\bs k!}\Zi kx + \sum_{\abs{\bs k} = \floor{\beta}} \frac{(\bs{x_j} - \bs x)^{\bs k}}{\bs k!}\Zi k{\tau_j}\Big)\bigg|\\     
		& = \frac1{\sqrt n}  \Abss{\sum_{\bs j = \bs 1}^{\bs p}\wjsb xhs\sum_{\abs{\bs k} = \floor{\beta}} \frac{(\bs{x_j} - \bs x)^{\bs k}}{\bs k!}\big(\Zi k{\tau_j} - \Zi kx\big)} \tag{by \ref{ass:weights:pol:rep}}\\
		& \leq \sum_{\abs{\bs k } = \floor{\beta}}\frac{\Csum\,M_i\,h^{\beta}}{\bs k!\,\sqrt n \,h^{\abs{\bs s}}}  \leq \Cenv\, M_i\, \frac{1}{{\sqrt n}\,h^{\abs{\bs s} - (\beta- \delta)}}\,.\tag{by \ref{ass:weights:sum}}
	\end{align*}
    
	Next, as for \textit{i)} let's turn to the Lipschitz-condition required in Lemma \ref{lem:manageability:lipschitz}. Again let $0< \epsilon< 1$, $\bs x ,\bs y$ with 
satisfying \eqref{eq:lipschitzcont} for some $\eta$ to be specified. 
    
	First consider $\epsilon > h$, and take $\eta=1$ in \eqref{eq:lipschitzcont}.

    As before, using a Taylor expansion up to the order $\floor{\beta}$ and \ref{ass:weights:pol:rep} for the weights we get that 
	\begin{align*}
		\sqrt n \absb{X_{ni}(\bs x)  - X_{ni}(\bs y)} & \leq \Abss{\sum_{\bs j = \bs 1}^{\bs p}\wjsb xhs\sum_{\abs{\bs k} = \floor{\beta}} \frac{(\bs{x_j} - \bs x)^{\bs k}}{\bs k!}\big(\Zi k{\tau_{x,j}} - \Zi kx\big)} \\
		& \qquad + \Abss{\sum_{\bs j = \bs 1}^{\bs p}\wjsb yhs\sum_{\abs{\bs k} = \floor{\beta}} \frac{(\bs{x_j} - \bs y)^{\bs k}}{\bs k!}\big(\Zi ky - \Zi k{\tau_{y,j}}\big)}\,,
	\end{align*}
	with $\bs{\tilde\tau_j} = \bs y + \tilde \theta_{\bs j}(\bs{x_j} - \bs y)$ for certain $\tilde\theta_j \in [0,1]$. Note that by property \ref{ass:weights:pol:rep} of the weights the intermediate term in \eqref{eq:hilfboundprocess} does not arise here since we expand to an order $< \abs{\bs s}$. The first term is upper-bounded by
	\begin{align*}
		& \sum_{\bs j = \bs 1}^{\bs p}\abs{\wjsb xhs}\, \sum_{\abs{\bs k} = \floor{\beta}} \frac{\abs{(\bs{x_j} - \bs x)^{\bs k}}}{\bs k!}\abs{\Zi k{\tau_{x,j}} - \Zi kx}\\
        &\leq \Cenv\, M_i\, \frac{1}{h^{\abs{\bs s}}}  h^{ \beta} \leq   \frac{\Cenv\, M_i\,}{h^{\abs{\bs s} - (\beta-\delta)}}  \epsilon^{\delta} \tag{since $\epsilon > h$}.
	\end{align*}

    If $\epsilon \leq h$ we can take the same bound as in the previous step \textit{i)}. Therefore, \eqref{eq:measucondprocess} is satisfied for $\kappa= \max(1/\delta,  2 + (d + \abs{\bs s}))$. The entropy integral in \eqref{eq:theconstant} is thus of order $\delta^{-1/2}$, and the maximal inequality \eqref{eq:maxboundpollard} yields the result.

\end{proof}

\subsection{Proof of Theorem \ref{thm:lower:bound:mean:derivatives}}\label{sec:prooflowerbound}

\begin{proof}[Proof of Theorem \ref{thm:lower:bound:mean:derivatives}]
For the proof we rely on a reduction to hypothesis testing, as detailed e.g.~in \citet[Section 2]{tsybakov2008introduction}.

    \textit{i)}  
	\textit{Lower bound of order $\bs p_{\min}^{-\alpha + \abs{\bs s}}$:} We construct two sequences of hypothesis functions $\mu_0 = 0$ and $\mu_1 = \mu_{1,p}$ such that $\mu_0, \mu_1 \in \mc H(\alpha, L)$,  $\norm{\dels \mu_0 - \dels \mu_1}_\infty \geq c\,\bs p_{\min}^{-\alpha + \abs{\bs s}}$ for a constant $c> 0$, and such that $\mu_1(\bs x_{\bs j}) = 0$ for all design points $\bs x_{\bs j}$, so that the distribution of the observations coincides for both hypotheses. 
    Set 
    $$p=\bs p_{\max},$$ which by assumption is of the same order as $\bs p_{\min}$.  
    For some fixed $1 \leq l \leq {\bs p}_{\min}$, using the design points $\bs x_{(l, \ldots, l)} = (x_{l,1}, \ldots, x_{l,d})$ and $\bs x_{(l+1, \ldots, l+1)} = (x_{l+1,1}, \ldots, x_{l+1,d})$ we set 
	$$ \mu_1(\bs x) = \tilde L (c\, p)^{-\alpha}\, \prod_{r =1}^d g\big( 2 \,c\, p\, (x_r - (x_{l,r} + x_{l+1,r})/2)\big)\,,\quad \bs x = (x_1,  \ldots, x_d)^\top \in \R^d, $$
	where $c=\max_k \fkmax$ and where the function $g\colon \R \to \R$ is given by
	$$ g(x) =  \exp\big( - 1/( 1 - x^2)\big)  \, \ind_{\abs{x} < 1} \,, \quad x \in \R\,. $$
    Since by Assumption \ref{ass:designdensity}, all design points are at a distance of at least $1/(c p)$ with $c=\max_k \fkmax$, it follows that $\mu_1$ indeed vanishes at the design points. 

    Turning to Hölder-smoothness, the $\nu^{\mathrm{th}}$-derivative of $g$ is given by
	$$ g^{(\nu)} (x) = \frac{\varphi_{3(\nu-1) + 1}(x)}{(1 - x^2)^{2\,\nu}}\, g(x)\,,\quad x\in \R\,,$$
	where $\varphi_{3(\nu-1) + 1}(x)$ is a polynomial of degree $3(\nu-1) + 1$. Thus, all derivatives of $g$ are non-zero and uniformly bounded, and Hölder-continuous of all orders.  
    
	The $\bs k^{\mathrm{th}}$ partial derivative of $\mu$ is given by 
	$$ \partial^{\bs k} \mu(\bs x) = \tilde L (c\, p)^{-\alpha + \abs{\bs k}}\, 2^{\abs{\bs k}} \prod_{r = 1}^d g^{(k_r)}\big(2 \,c\, p\, (x_r - (x_{l,r} + x_{l+1,r})/2) \big)\,.$$
    Together with smoothness of the function $g$ and its derivatives this implies that $\mu_{1,\bs p} \in \mc H(\alpha, L)$, where $\tilde L$ in the definition of $\mu_{1, \bs p}$ can be adjusted to obtain a Hölder norm $ \leq L$.  
	Finally,  
	\begin{align*}
		\norm{\dels \mu_0 - \dels \mu_1}_\infty  = \norm{\dels \mu_1}_\infty & =  \tilde L (c\, p)^{-\alpha + \abs{\bs s}}\, 2^{\abs{\bs s}} \prod_{r = 1}^d \sup_{x \in T_{l,r}} \big| g^{(s_r)}(x) \big|\,\\
        & \geq c_1 \, p^{-\alpha + \abs{\bs s}}.
	\end{align*}
	for $c_1>0$, with $T_{l,r} = \big[(x_{l,r} + x_{l+1,r})/2 - 1/(2 c p), (x_{l,r} + x_{l+1,r})/2 + 1/(2 c p)\big]$. 
	\\

    	\textit{ii) Lower bound of order $(\log(n\bs{p^1})/(n\bs{p^1}))^{(\alpha - \abs{\bs s}) / (2\alpha + 1)}$.} This is derived in analogy to the case for estimating the mean function, by using the same hypothesis functions with the same bounds on the Kullback distance. In the notation of the argument, the distance between  the derivatives of the hypotheses functions is of order $h_{n,\bs p}^{\alpha - \abs{\bs s}}$, which gives the required rate.

    	\smallskip

	    \textit{Lower bound of order $n^{-1/2}$.} First assume that the errors $\epsilon_{i,\bs j}=0$. 
        Consider the hypothesis functions $\mu_0 = 0$ and $\mu_1 (\bs x) = \bs x^{\bs s}/\sqrt n$ for which $\norm{\mu_0^{(\bs s)} - \mu_1^{(\bs s)}}_\infty = \bs s!/\sqrt n$, and take $Z_i (\bs x) = N_i\, \bs x^{\bs s}$, where the $N_i \sim \mc N(0, 1)$ are independent standard normal random variables. The observations are then given by $Y_{i,\bs j}^{(1)} = \bs x_{\bs j}^{\bs s} \, (N_i + 1/\sqrt n)$ and $Y_{i,\bs j}^{(0)} = \bs x_{\bs j}^{\bs s}\,  N_i$, so that the observations in each row are of the forms $Y_{i,\cdot}^{(1)} = \bs a^\top  \, (N_i + 1/\sqrt n)$ and $Y_{i,\cdot}^{(1)} = \bs a^\top  \, N_i$ for a $\bs p^{\bs 1}$-dimensional vector $\bs a$.         
        Now the Kullback-divergence between $N_i$ and $N_i+ 1/\sqrt n$ is $(2\,n)^{-1}$ and the Kullback-divergence between the $n$ i.i.d.~repetitions therefore is equal to $1/2$. By Lemma \ref{lem:KLdivergence}, the Kullback-divergence between $Y_{i,\cdot}^{(0)}$ and  $Y_{i,\cdot}^{(1)}$ is equal to that between $N_i$ and $N_i+ 1/\sqrt n$, so that the Kullback-divergence between the sets of observations is equal to $1/2$ as well. Since the KL-divergence only decreases after adding the normally distributed observational noise  $\epsilon_{i,\bs j}=0$, the conclusion follows. 
        

  	\smallskip

    \textit{iii)} 
    By taking functions and processes which only depend on a single coordinate, and taking all $\abs{\bs s}$ partial derivatives in that direction we may reduce the analysis to the one-dimensional case $d=1$. Thus $s\in \N$ is a  scalar in what follows.  
    To define the hypotheses functions, consider the Riemann-Liouville fractional integral of order $\beta>0$ of a function $f$,
    \begin{equation}\label{eq:RLfI}
    I_{0+}^{\beta}f(t) = \frac{1}{\Gamma(\beta)}\, \int_0^t (t-s)^{\beta-1}\, f(s)\, \dx s, \qquad t \in [0,1].
    \end{equation}
    Moreover, consider the fractional derivative of order $\beta$: For $n \in \N_0$ let $D^n$ denote $n$th order ordinary derivative. Further let $m=[\beta]$ denote the integer part of $\beta$ (largest integer $\leq \beta$).  Then
    $$ D_{0+}^{\beta}f(t) := D^{[\beta] + 1} I_{0+}^{[\beta] + 1 - \beta} f(t)$$
    if it exists. The operators $D_{0+}^{\beta}$ and $I_{0+}^{\beta}$ are inverse operators, see  \citet[Lemma 2.2.5]{balachandran2023introduction}. 
    Also recall the Riemann-Liouville $\beta$-fractional Brownian motion in \eqref{eq:RLfB1}:
    \begin{equation}\label{eq:RLfB}
    R_t^\beta =  \int_0^t (t-s)^{\beta-1/2}\,  \dx W_s,
    \end{equation}
    where $(W_s)$ is a standard Brownian motion. 

    Now we derive the lower bound in the continuous model 
    $$Y^{j,n}(t) = \mu_{j,n}(t) + Z(t), \qquad j=0,1,$$ without errors. Since this is more informative than the discrete model with observational error, the conclusion also follows for the discrete model.
    Moreover, we can adjust for the Hölder constants $L$ and $C_Z$ by multiplying by a sufficiently small number, without effecting the rates. 

    As hypothesis functions we take $\mu_0=0$, and for $\mu_1 = \mu_{1,n}$ consider 
    $$f(t) = \exp\big(1/(t\, (1-t))\big) \, 1_{(0,1)}(t), \qquad g(t) = D_{0+}^{\beta_1 +1/2}\,f (t).$$
       Let $h=n^{-1/(2\,(\, \alpha - \beta_1))}$, and take 
      $$\mu_{1,h}(t) = h^{\alpha}\, f(t/h).$$ 
    Then as shown previously, $\mu_{1,h}$ is $\alpha$-Hölder continuous and moreover, the sup-norm distance between $\mu_0^{(s)}$ and $\mu_{1,h}^{(s)}$ 
    is of order $h^{\alpha - s} = n^{-\frac{\alpha - s}{2\,(\, \alpha - \beta_1)}}$, as required.  
    Further, since $f(t) = I_{0+}^{\beta_1 +1/2}g (t)$ we have 
    \begin{equation}\label{eq:sechypothfct}
        \mu_{1,h}(t) = I_{0+}^{\beta_1+1/2} \big(h^{\alpha - \beta_1-1/2}\, g(\cdot/h) \big) (t).
    \end{equation} 
    Next take $Z_i$ independent Riemann-Liouville $\beta_1$-fractional Brownian motions $R_t^{\beta_1}$.
    From \citet[Lemma 2]{schmidt2014asymptotic} the Kullback-Leibler divergence between between the distributions of $Y^{0}$ and $ Y^{1,n}$ is equal to half the squared distance $ \|\cdot \|^2_{\mathbb H}$ in the reproducing kernel Hilbert space associated with the Riemann-Liouville $\beta_1$-fractional Brownian motion, which is computed in \citet[Lemma 10.2]{van2008reproducing}. Therefore, it follows that for the Kullback-Leibler divergence, 
    \begin{align*}
        \KL(Y^{0}, Y^{1,n}) & = \frac12\, \|\mu_{1,h}\|^2_{\mathbb H}\\
        & = \frac12\,\frac{h^{2(\alpha - \beta_1)}\, h^{-1}}{\Gamma(\beta_1 + 1/2)^2}\, \int_0^1\, g^2(t/h)\, \dx t \tag{using \eqref{eq:sechypothfct}}\\
        & = \frac{\int_0^1\, g^2(t)\, \dx t}{\Gamma(\beta_1 + 1/2)^2}\, n^{-1}.
    \end{align*}
    as required for $n$ independent and identically distributed repetitions. 

    Finally, the Riemann-Liouville $\beta_1$-fractional Brownian motions $Z:= R_t^{\beta_1}$ is a Gaussian process. We need to show that it is $\beta$-Hölder smooth for $\beta < \beta_1$ with square-integrable Hölder norm. 
    To this end we apply \citet[Theorem 8, (1)]{da2023sample} together with \citet[Theorem 1]{azmoodeh2014necessary}. 
    The covariance of $R_t^{\beta_1}$ can be written as
    $$\gamma_{\beta_1}(s,t) = \expec[R_t^{\beta_1} R_s^{\beta_1}] = \int_0^{\min(s,t)}\, (s-u)^{\beta_1 - \frac12}\,(t-u)^{\beta_1 - \frac12}\, \dx u = \Gamma(\beta_1 + 1/2)\, I_{0+}^{\beta_1 + \frac12} f_{s, \beta_1}(t),$$
    where $f_{s, \beta_1}(u) = (u-s)_+^{\beta_1 - \frac12}$. Now using \citet[Theorem 2.2.3]{balachandran2023introduction} we obtain
    $$ \partial_t^{\floor{\beta_1}}\gamma_{\beta_1}(s,t) = \Gamma(\beta_1 + 1/2)\, I_{0+}^{\beta_1 - \floor{\beta_1} + \frac12} f_{s, \beta_1}(t) = \frac{\Gamma(\beta_1 + 1/2)^2}{\Gamma(\beta_1 - \floor{\beta_1} + 1/2)}\, I_{0+}^{\beta_1  + \frac12} f_{t, \beta_1- \floor{\beta_1}}(s), $$
    where the second equality follows from the definition of $f_{s, \beta_1}(u)$ and the fractional integral. Then
        \begin{align*}
         \partial_t^{\floor{\beta_1}}\partial_s^{\floor{\beta_1}}\gamma_{\beta_1}(s,t) & =  \frac{\Gamma(\beta_1  + 1/2)^2}{\Gamma(\beta_1 - \floor{\beta_1} + 1/2)}\, I_{0+}^{\beta_1 - \floor{\beta_1} + \frac12} f_{t, \beta_1- \floor{\beta_1}}(s)\\
         & = \frac{\Gamma(\beta_1  + 1/2)^2}{\Gamma(\beta_1 - \floor{\beta_1} + 1/2)^2}\,\int_0^{\min(s,t)}\, (s-u)^{\beta_1 - \floor{\beta_1} - \frac12}\,(t-u)^{\beta_1 - \floor{\beta_1} - \frac12}\, \dx u,
        \end{align*}
        which is proportional to the covariance function of the Riemann-Liouville $\beta_1- \floor{\beta_1}$-fractional Brownian motion. Now its variogram is bounded by $C\, |t-s|^{2\, \beta_1- \floor{\beta_1}}$, see \citet{ichiba2022path}, Theorem 3.1 and Remark 3.1, with $\gamma=0$ and $\alpha = \beta_1- \floor{\beta_1} - 1/2$.  Hölder continuity follows from \citet[Theorem 1]{azmoodeh2014necessary}.   
\end{proof}

\subsection{Proof of Theorem \ref{thm:asymptotic:normality}}

\begin{proof}[Proof of Theorem \ref{thm:asymptotic:normality}]

		From Theorem \ref{thm:mean:upper_bounds} together with the assumptions on $\bs p$ and $n$, for the choice $h= h_n  \in H_n$ of the smoothing parameters it follows that for the terms $I_{1}^{\bs p,h, \bs s}(\bs x)$  and $I_{2}^{n, \bs p,h, \bs s}(\bs x)$ in the error decomposition \eqref{eq:decomposition} we have that
        %
	\begin{align*}
		\sup_{\mu \in \hclass } \sqrt n \, \normb{I_{1}^{\bs p,h, \bs s}}_\infty \to  0 \quad \text{and} \quad \sqrt n \, \normb{I_{2}^{n, \bs p,h, \bs s}}_{\infty} \stackrel{\prob}{\to}  0 \,.
	\end{align*}
	We shall apply the functional central limit Theorem in \citet[Theorem (10.6)]{pollard1990empirical}	to the triangular array of processes $(X_{ni})_{n \in \N, 1 \leq i \leq n}$ given by \eqref{eq:X:process:deriv} 
	to obtain 
	\begin{align}\label{eq:clthelp}
		S_n : = \sum_{i=1}^n X_{ni} = \sqrt n \, I_{3}^{n, \bs p,h, \bs s}   \stackrel{D}{\longrightarrow} \ \mc G\big(\bs 0,\partial^{(\bs s, \bs s)^\top} \Gamma\big) \, .
	\end{align}
	With Slutsky's Theorem \citep[Example 1.4.7]{van1996weak} we then obtain the assertion of Theorem \ref{thm:asymptotic:normality}. 
    
    For \eqref{eq:clthelp} we have to check the conditions i) - v) of \citet[Theorem (10.6)]{pollard1990empirical}. 

First note that since $Z_1 \in \mc P(\beta, C_Z)$ with $\beta > \abs{\bs s}$, using the argument for Lemma \ref{lem:mean:estimation:rates:asbefore}, \textit{i)} we obtain
\begin{equation}\label{eq:taylorprocess}
    \sup_{\bs x \in [0,1]^d} \Big|\sum_{\bs j =\bs 1}^{\bs p} \wjsb xhs  Z_1(\bs{x_j})  - \dels Z_1(\bs x) \Big| \leq M_1 \, h^{\min(1,\beta - \abs{\bs s})}\, \sum_{|\bs k|= \abs{\bs s}}\frac{C_1}{\bs k !}
\end{equation}
    
	Concerning i) in Pollard's theorem:  in the proof of Lemma \ref{lem:mean:estimation:rates:new}, part \textit{i)}, \eqref{eq:bound:smooth:process} we have shown that the processes $X_{n}\s = (X_{n1},\ldots,X_{nn})$ satisfy the properties a) and b) of Lemma \ref{lem:maximalInequalityExpection} with respect to the envelope $\bs \Phi_{n} \defeq 2\, \Cenv\, (M_1, \ldots , M_n)^\top / \sqrt n$. 

	For ii) we calculate $\lim_{n\to \infty}\expec[S_n(\bs x)S_n(\bs x^\prime)]$.  Using \eqref{eq:taylorprocess} and 
$\sqrt n \abs{X_{n1}(\bs x)}  \leq 2\, \Csum\, M_1$ we obtain
	\begin{align*}
		\expec[S_n(\bs x)S_n(\bs x^\prime)] & = \expec\Big[\sum_{\bs j =\bs 1}^{\bs p} \wjsb xhs  Z_1(\bs{x_j}) \sum_{\bs k = \bs 1}^{\bs p}\wksb{x^\prime}hs Z_1(\bs{x_k})\Big] \\
		&  \stackrel{n \to \infty}{\longrightarrow} \expec\big[ \dels Z_1(\bs x)\,\dels Z_1(\bs x^\prime)\big]  = \partial^{(\bs s,\bs s)^\top}\Gamma(\bs x, \bs x^\prime)\,, 
	\end{align*}
	uniformly for all $\bs x, \bs x^\prime \in [0,1]^d$.  
	
	For iii) and iv): The coordinates of the envelope $\bs \Phi_{n}$ are given by $ 2\, \Cenv\, M_i / \sqrt n$. So this holds by assumption of square-integrability of the $M_i$.
	
	For iv) (Lindeberg-condition for the envelope): This is immediate since the coordinates of $\bs \Phi_{n} $ are i.i.d.. 
	
	For v) we calculate similarly as for ii)
	\begin{align*}
		\rho_n^2(\bs x, \bs x^\prime)  
		& = \expec\Bigg[\sum_{\bs j=\bs 1}^{\bs p} \wjsb xhs  Z_1(\bs{x_j}) \sum_{\bs k = \bs 1}^{\bs p}  \wksb xhs Z_1(\bs{x_k})\\
		& \Quad - 2\sum_{\bs j=\bs 1}^{\bs p} \wjsb {x^\prime}hs  Z_1(\bs{x_j}) \sum_{\bs k = \bs 1}^{\bs p}  \wksb xhs Z_1(\bs{x_k}) \\
		& \Quad + \phantom{2}\sum_{\bs j=\bs 1}^{\bs p} \wjsb {x^\prime}hs  Z_1(\bs{x_j}) \sum_{\bs k = \bs 1}^{\bs p}  \wksb {x^\prime}hs Z_1(\bs{x_k})\Bigg] \\
		&  \stackrel{n \to \infty}{\longrightarrow} \partial^{(\bs s, \bs s)^\top}\Gamma(\bs x,\bs x) -  2\,\partial^{(\bs s, \bs s)^\top}\Gamma(\bs x,\bs x^\prime) +\partial^{(\bs s, \bs s)^\top}\Gamma(\bs x^\prime, \bs x^\prime) \defeql \rho(\bs x, \bs x^\prime)\,
	\end{align*}
	uniformly for all $\bs x, \bs x^\prime \in [0,1]^d$ by \ref{ass:weights:pol:rep}. Now let $(\bs x_n, \bs y_n)_{n \in \N} \subset [0,1]^d$ be a deterministic sequence for which $\rho(\bs x_n, \bs y_n) \to 0, n\to \infty$. Then 
	\begin{align*}
		0 \leq \rho_n(\bs x_n, \bs y_n) & \leq \sup_{\bs x, \bs x^\prime \in [0,1]^d} \absb{ \rho_n(\bs x, \bs x^\prime) - \rho(\bs x, \bs x^\prime)}+ \rho(\bs x_n, \bs y_n) \stackrel{n \to \infty}\longrightarrow 0\,.
	\end{align*}
	\end{proof}

\subsection{A maximal inequality from \citet{pollard1990empirical}} \label{sec:maxinequa}

We recall from \citet[Section 7]{pollard1990empirical} the maximal inequality (7.10) which is used in the proof of Lemma \ref{lem:mean:estimation:rates:new} as well as in Pollard's CLT, \citet[theorem (10.6)]{pollard1990empirical}, on which our Theorem \ref{thm:asymptotic:normality} is based. 

Consider a triangular array $(X_{n,i})_{n \in \N, 1\leq i \leq k_n}$, 
$$ X_{n,i} \colon \Omega\times [0,1]^d \to \R$$ of stochastic processes with independent rows, assumed to be centered, $\expec\big[X_{n,i}(\bs x)\big]=0$, $\bs x \in [0,1]^d$, and set 
	\begin{equation*}
		S_n(\bs x) \defeq \sum_{i=1}^{k_n} X_{n,i}(\bs x),
	\end{equation*}
assuming the expected values exist. 
The sequence $\bs \Phi_n$ of $k_n$-dimensional random vectors is called an \textit{envelope} of $(X_{n,i})$ if $\bs \Phi_n(\omega)$ is an envelope of the set 
$$F_n(\omega)\defeq \big\{(X_{n,1}(\omega, \bs x) ,\ldots,X_{n,k_n}(\omega, \bs x))^\top \mid \bs x \in [0,1]^d\big\}$$
for $n  \in \N$ and $\omega \in \Omega$. Here, a vector $ \bs\Phi = (\Phi_1,\ldots, \Phi_n)^\top\in \R^n$ is an \textit{envelope} for the bounded subset $B \subset \R^n, n\in\N$ if $\abs{b_i} \leq \Phi_i$ for all $\bs b = (b_1,\ldots, b_n)^\top \in B$.

Let us state the maximal inequality \citet[Section 7, display (7.10)]{pollard1990empirical}. 
\begin{lemma}\label{lem:maximalInequalityExpection}
	For the triangular array $(X_{n,i})$ with envelope $(\Phi_n)$, assume that there exists a deterministic function $\lambda\colon[0,1] \to  \R$, the \textit{capacity bound}, for which
\begin{enumerate}
	\item[a)] $\int_0^1 \sqrt{\log ( \lambda(x))}\dx \varepsilon < \infty,$
	\item[b)] $ D(\varepsilon\norm{\bs \alpha \circ \Phi_n(\omega)}_2, \bs \alpha \circ F_n(\omega)) \leq \lambda(\varepsilon)$  for $\varepsilon \in [0,1]$ and all $\bs \alpha \in \R_+^n$, $n \in \N$ and all $\omega\in \Omega$
\end{enumerate}
where $\bs x \circ \bs y$ denotes the Hadamard (component-wise) product of two vectors of equal dimension, and  $D(\varepsilon\norm{\bs \alpha \circ \Phi_n(\omega)}_2, \bs \alpha \circ F_n(\omega))$ is the $\varepsilon\norm{\bs \alpha \circ \Phi_n(\omega)}_2$-packing number of the set $\alpha \circ F_n(\omega) $. 
    
Then for each $1\leq p <\infty$ there exists a constant $\tilde C_p>0 $ such that 
	\begin{equation}\label{eq:maxboundpollard}
		\expec\big[\norm{S_n }^p_\infty\big] \leq \tilde C_p \, \Big(\int_0^1 \sqrt{\log ( \lambda(\varepsilon))}\dx \varepsilon\Big)^p \expec\big[\norm{\bs \Phi_n}_2^p\big].
	\end{equation}
\end{lemma}

The requirements a) and b) of the lemma are satisfied for processes with Hölder-continuous paths, as made precise in the following lemma which is \citet[Lemma 3]{berger2023dense}.
\begin{lemma}\label{lem:manageability:lipschitz}
	Consider the triangular array $(X_{n,i})_{n\in \N, 1\leq i \leq k_n}$ of stochastic processes on $[0,1]^d$ with a suiting sequence of envelopes $(\bs \Phi_n)_{n \in \N}$. If there are constants $K, \kappa_1, \kappa_2 >0$ such that for all $\bs x,\bs x^\prime \in [0,1]^d, n\in\N $ and $\epsilon>0 $ it holds that
	\begin{equation}\label{eq:measucondprocess}
		\norm{\bs x - \bs x^\prime}_2 \leq \epsilon^{\kappa_1} \Rightarrow \big|X_{n,i}(\bs x) - X_{n,i}(\bs x^\prime)\big| \leq \epsilon^{\kappa_2} K \Phi_{n,i}, 
	\end{equation}
	where $i =1,\ldots, k_n$, and $\bs \Phi_n = (\Phi_{n,1},\ldots, \Phi_{n,k_n})^\top$, then setting $\kappa=\kappa_1/\kappa_2$ we have that
	\begin{equation*}
		D(\epsilon \norm{\bs \alpha \circ \bs \Phi_n}_2, \bs \alpha  \circ F_n) \leq \bigg(K^\kappa \, \epsilon^{-\kappa} + 2\bigg)^d =: \lambda_\kappa(\varepsilon).
	\end{equation*}
	Further it holds that 
	\begin{equation}\label{eq:theconstant}
		\Lambda_\kappa \defeq \int_0^1 \sqrt{\log(\lambda_\kappa(\varepsilon))}\, \dx \epsilon = \mathcal O\big( \sqrt \kappa\big).
	\end{equation}
\end{lemma}
%

\section*{Acknowledgements}

HH gratefully acknowledges financial support from the DFG, grant HO 3260/9-1.

\bibliographystyle{Chicago}
\bibliography{Literature_fda}

\appendix


\section{Further proofs}\label{app:additional_proofs}

\begin{proof}[Proof of Lemma \ref{lem:mean:estimation:rates:asbefore}, \textit{ii)}]

	We apply Dudley's entropy bound \citep[Corollary 2.2.8]{van1996weak}. To this end, note that by Assumption \ref{ass:model}, $\bar\e_{\bs 1},\ldots,\bar \e_{\bs p}$ are independent, centered, sub-Gaussian random variables with parameters bounded by $\zeta^2\sigma^2/n>0$ and $\E[\bar \e_{\bs j, n}^2]=\sigma_{\bs j}^2/n$, $\bs j=\bs 1,\dotsc,\bs p$. 
	Therefore, the process
	$$S_{n,\bs p, h}(\bs x) = \sqrt{n\bs{p^1}h^{d-2\abs{\bs s}}}\,{I_2^{n,\bs p, h}}(\bs x)$$ 
	is a sub-Gaussian process w.r.t.~the semimetric 
	\begin{align}
		\dd_{S}^2(\bs x, \bs y) & \defeq  \zeta^2\, \sigma^2\, \bs{p^1}h^{d-2\abs{\bs s}} \sum_{\bs j = \bs 1}^{\bs p}\big( \wjb{x}{h} -\wjb{y}{h} \big)^2  \nonumber \\
		&\leq M^2 \bigg(\frac{\norm{\bs x - \bs y}_\infty}{h}\wedge 1\bigg)^2
		\leq M^2\,, \label{eq:metric:Stilde}
	\end{align}
	where $M = \sqrt{2\Ccard} \Clip \, \zeta\, \sigma$, and we used the Lipschitz continuity of the weights, \ref{ass:weights:lipschitz}.  Now, given $0 < \delta < M$ to bound the covering number $	N\big( [0,1]^d,\delta;\dd_S \big)$ and hence the packing number $	D\big( [0,1]^d,\delta;\dd_S \big)$ of $[0,1]^d$ w.r.t.~$\dd_S$, note that from \eqref{eq:metric:Stilde}, 
	\begin{align*}
		\dd_S(\bs x, \bs \tau) \leq \delta \qquad \Longrightarrow  \qquad \norm{\bs \tau_{j}-\bs x}_\infty\leq \frac{\delta h}M, \qquad \bs x, \bs \tau \in [0,1]^d.
	\end{align*}
	We obtain
	\begin{equation}
		D\big( [0,1]^d,\delta;\dd_S \big) \leq N\big( [0,1]^d,\delta/2;\dd_S \big) \leq  \bigg(\frac {2M}{\delta {h}} \bigg)^d . \label{eq:bound:packing}
	\end{equation}
	By observing from \eqref{eq:metric:Stilde} that the diameter of $[0,1]^d$ under $\dd_S$ is upper bounded by $2 M$, Dudley's entropy bound implies that 
	\begin{align} \label{eq:integral:packing}
		\E\Big[\sup_{\bs x \in [0,1]^d} \big|S_{n,\bs p, h}(\bs x)\big| \Big] \leq \E\Big[ \big| S_{n,\bs p, h}(\bs x_0) \big| \Big]+ C \int_0^{2 M} \sqrt{\log \Big( D\big( [0,1]^d,\delta;\dd_S \big) \Big)} \,\dx \delta
	\end{align}
	for fixed  $\bs x_0 \in [0,1]^d$ and a universal constant $C>0$. 	Using \eqref{eq:bound:packing}, the integral in the second term is bounded by 
	\begin{align*}
		 \int_0^{2M}\sqrt{\log\big((2M / (\delta h)^d) \big)} \dx \delta
		 &= \frac{2 M\,d}{h} \int_0^{h} \sqrt{\log(\delta^{-1})} \,\dx \delta \\
		 &\leq  2\,  M\,d \bigg(\sqrt{-\log (h)}-\frac{1}{2\sqrt{-\log (h)}}\bigg) 
		=\Oop{ \sqrt{-\log (h)} } \,.
	\end{align*} 
	For the first term in \eqref{eq:integral:packing}, apply Jensen's inequality to obtain
	\begin{align*}
		\E\Big[ \big| S_{\bs p}(\bs x_0) \big| \Big]^2 &\leq \E\big[\widetilde{S}_{\bs p}^2(\bs x_0)\big] \leq \sigma^2 \bs{p^1}h^{d-2\abs{\bs s}} \sum_{\bs j = \bs 1}^{\bs p} \wjsb{x_0}{h}s^2 \leq \sigma^2 \Csum \Cmax < \infty
	\end{align*}
	by 
	\ref{ass:weights:max} and \ref{ass:weights:sum}. This concludes the proof.  	 	
\end{proof}

\begin{proof}[Proof of Theorem \ref{thm:lower:bound:mean:derivatives}, \textit{ii)}, first bound]
        Similarly as for the mean function itself we set $Z = 0\;(\in \Pclass)$ to adapt the argument in \citet[Theorem 2.10]{tsybakov2008introduction} and for the distribution for the errors we can assume $\mathcal N(0, \sigma_0^2/n)$.  
	For sufficiently small constants $c_0, c_1>0$ to be specified let 
	\begin{align*}
		N_{n,p} = \Bigg\lceil c_0 \bigg(\frac{n\,\bs{p^1}}{\log(n\,\bs{p^1})}\bigg)^{\frac{1}{2\alpha + d}}\Bigg\rceil, \quad h_{n,\bs p} = N_{n,\bs p}^{-1}, \quad s_{n,\bs p} = c_1\,h_{n,\bs p}^{\alpha - \abs{\bs s}}\,.
	\end{align*}
	We use similar hypothesis functions as in the previous case. Let
	\begin{align*}
		\tilde g(x) & = \exp\bigg(-\frac1{1-x^2}\bigg) \ind_{\abs{x} < 1}\,, \quad x\in \R\,,
	\end{align*}
	and for $\bs l  = (l_1, \ldots, l_d)^\top \in \{1,\ldots,N_{n,p}\}^d$ set the hypothesis functions as
	\begin{align*}
		\mu_{\bs l}(x) & = \tilde L \, h_{n,p}^\alpha \prod_{r = 1}^d \tilde g\bigg( \frac{2(x_r - z_{l_r})}{h_{n,\bs p}}\bigg) \quad \text{with} \quad  z_{l_r}  = \frac{l_r - 1/2}{N_{n,\bs p}}\,.
	\end{align*}
	In the following we let $\tilde x_{r, l_r} = 2(x_r - z_{l_r})/h_{n,\bs p}$. With the chain rule the derivatives of $\mu_{\bs l}$ are of a similar form as in the previous part and given by
	\begin{align*}
		\dels \mu_{\bs l}(\bs x) & = \tilde L\,h_{n, \bs p }^{\alpha} \prod_{r = 1}^d \tilde g^{(s_r)} (\tilde x_{r, l_r}) \\
		& =  2^{-\abs{\bs s}}\,\tilde L\, h_{n,\bs p}^{\alpha - \abs{\bs s}} \prod_{r = 1}^d \frac{\varphi_{3(s_r - 1)+1}(\tilde x_{r, l_r})}{(1-\tilde x_{r, l_r}^2)^2} \tilde g(\tilde x_{r, l_r})\,,
	\end{align*}
	where again $\varphi_{3(s_r-1)+1}(x)$ are polynomials of degree $3(s_r-1) + 1$ and not equal to zero. The hypothesis functions as well as their derivatives have disjoint supports and we obtain $\norm{\dels \mu_l - \dels \mu_r}_\infty = \norm{\dels \mu_l}_\infty \geq \text{const.}\,s_{n,p}$. In order to obtain $s_{n,p}$ as a lower bound via Proposition 2.3 in \cite{tsybakov2008introduction} we need to show that for a $\kappa \in (0, 1/8)$ and $p$ and $n$ being large enough it holds 
	\begin{align*}
		\frac{1}{N_{n,\bs p}^d} \sum_{\bs l} \mathrm{KL}(P_{\mu_{\bs l}}, P_0) \leq d\, \kappa \,\log\big(N_{n,\bs p}\big)\,,
	\end{align*}
	where $P_{\mu_l}$ is the joint normal distribution being generated with the mean function $\mu_l$. This is done by
	\begin{align*}
		\frac{1}{N_{n,\bs p}^d} \sum_{\bs l}  \KL(\prob_{\mu_{\bs l}},\prob_0) 
		&= \frac{1}{N_{n,\bs p}^d}\, \frac{n}{2 \sigma_0^2}\,\sum_{\bs l}\, \sum_{\bs 1\leq \bs j \leq \bs p}\,  \mu_{\bs l}^2(\bs{x_j}) \tag{normal distr.}\\
		&\leq \frac{n}{2 \sigma_0^2}\, \frac{\tilde L^2}{e^{2\,d}}\, h_{n,\bs p}^{2\alpha+d} \sum_{\bs 1\leq \bs j \leq \bs p}\,\sum_{\bs l}\,\prod_{r = 1}^d\ind_{\abs{2(x_{j_r} - z_{l_r})} <h_{n,\bs p} } \tag{Def. $\mu_{\bs l}$}\\
		& \leq \frac{n}{2 \sigma_0^2}\, \frac{\tilde L^2}{e^{2\,d}}\, h_{n,\bs p}^{2\alpha+d}  \,\bs{p^1}\, \tag{disjoint supp.} \\
		& \leq \text{const.}\,  \log(n\, \bs p^{\bs 1}). \tag{choice of $h_{n,\bs p}$}
	\end{align*}

    

	%
\end{proof}

\begin{proof}[Proof of Lemma \ref{lemma:locpol:weights}]
    Analogously to \citep[Formula (31)]{berger2023dense} we can write 
    	\begin{equation*}
		{\widehat{\mu}_{n,\bs p,h}^{\,(\bs s)}}(\bs x) =  \sum_{\bs j=\bs 1}^{\bs p} \wjsb xhs \, \oY{\bs j}
	\end{equation*}
    as a linear estimator with
	\begin{align} \label{eq:weights:locpol}
		\wjsb xhs \defeq \frac1{\bs p^{\bs 1}h^{d +\abs{\bs s} }} \dels U_m(\bs 0)^\top \, B_{\bs p,h}^{-1}(\bs x) \, U_{m,h}(\bs{x_j}-\bs x)\,  K_h(\bs x_{\bs j}-\bs x)\,,
	\end{align}
	and 
    	\begin{equation}
		B_{\bs p,h}(x) \defeq \frac1{\bs{p^1}h^d}\sum_{\bs j=\bs 1}^{\bs p} U_{m,h}(\bs{x_j}-\bs x)  \, U_{m,h}^\top(\bs x_{\bs j}-\bs x) \, K_{h}(\bs x_{\bs j}-\bs x)  \quad \in\R^{N_{m,d}\times N_{m,d}}\,.\label{eq:Bp}
	\end{equation}
    Now by \citep[Lemma 8]{berger2023dense} there exist a sufficiently large $c\in \R_+$, $p_0 \in \N$ and a sufficiently small $h_0 \in \R_+$ such that for all $\bs p$ with $\bs p_{\min} \geq p_0$ and $h \in (c/{\bs p}_{\min}, h_0]$ the smallest eigenvalues $\lambda(B_{\bs p, h}(x))$ of the matrix $B_{\bs p, h}(x)$ are bounded from below by a universal constant $\lambda_0$ for any $\bs x\in [0, 1]^d$. This implies the invertibility of the matrix $B_{\bs p, h}(x)$. As a consequence the local polynomial estimator in \eqref{eq:lp:mu:s} with any order $m \in \N$ is unique and a linear estimator with weights given in \eqref{eq:weights:locpol}. 
    We show that the local polynomial estimator defined in \eqref{eq:locpol} fulfills \ref{ass:weights:pol:rep}, meaning that estimator reproduces the $\bs s^{\mathrm{th}}$-derivative of polynomials up to degree $\alpha$. The remaining Assumptions follow analogously to \citep[Lemma 1]{berger2023dense} by multiplying he weights with $h^{-\abs{\bs s}}$ and using $\norm{\dels U_m(\bs 0)}_2 = \norm{U_m(\bs 0)}_2 = 1$.  \\
    
    Ad \ref{ass:weights:pol:rep}. Using the Taylor expansion for polynomial $Q$ of order $r$ with $r = \floor{\alpha}$
	\begin{align*}
		Q(\bs{x_j}) & = \sum_{ \abs{\bs k} = 0}^{r} \partial^{\bs k}Q(\bs x) \frac{(\bs{x_j}- \bs x)^{\bs k}}{\bs k!} \\
		& = \big(\underbrace{Q(\bs x), \partial^{\bs \psi(1)} Q(\bs x)\, h^{\abs{\bs \psi(1)}}, \ldots, \partial^{\bs \psi(N_{r, d-1})} Q(\bs x)\,h^{\abs{\bs \psi(N_{r, d-1})}}}_{\defeql q(\bs x)^\top}\big)\, U_{r, h}(\bs{x_j} - \bs{x})\,.
	\end{align*}
	Now plugging $Q(\bs{x_j})$ for $\bar Y_{\bs j}$ into the estimator gives 
	\begin{align*}
		\hat\nu_{n,p,h}(\bs x;Q(\bs x_j)) & = \argmin_{\bs \nu \in \R^{N_{r, d}}} \sum_{ \bs 1 \leq \bs j \leq \bs p} \big( Q(\bs{x_j}) - \bs \nu^\top U_{r, h}(\bs{x_j} - \bs{x}) \big)^2 K_h(\bs{x_j} - \bs{x}) \\
		& =\argmin_{\bs \nu \in \R^{N_{r, d}}} \sum_{ \bs 1 \leq \bs j \leq \bs p} \big(( q(\bs x) - \bs \nu )^\top U_{r, h}(\bs{x_j} - \bs{x})\big)^2 K_h(\bs{x_j} - \bs{x}) = q(\bs x)\,.
	\end{align*}
	Since 
	$\hat \nu_{n,p,h}(\bs x;Q(\bs x_j))_{\psi^{-1}(\bs s)} = \dels Q(\bs x) h^{\abs{\bs s}}$
	it is 
	$$\hat \mu^{(\bs s)}_{n,\bs p,h}(\bs x; h; Q(\bs{x_j})) = h^{-\abs{\bs s}}\,\hat \nu_{n,p,h}(\bs x;Q(\bs x_j))_{\psi^{-1}(\bs s)} = \dels Q(\bs{x_j})$$
	and estimator reproduces the desired derivative.
	Therefore the linear representation of the estimator gives
	\begin{align*}
		\sum_{\bs 1 \leq \bs j \leq \bs p} Q(\bs{x_j}) \wjsb xhs = \dels Q (\bs x)\,.
	\end{align*}
	For the equivalent representation in \eqref{eq:weights:rep:pol:analytic} set $Q = (\bs{x_j} - \bs \cdot)^{\bs r}$ with $\dels Q(\bs x) = \bs r \cdots (\bs r - \bs s) (\bs{x_j} - \bs x)^{\bs r - \bs s} $. Plugging this into the last equation shows that for $\bs s = \bs r$ the right hand side has value $\bs s!$. If there is an $s_i > r_i$ for any $i = 1, \ldots, d,$ then the right hand side is equal to zero. If $s_i < r_i$ for any $i = 1,\ldots, d,$ then the left hand side is zero. Since both sides are equal we have shown $\Rightarrow$. Since $(\bs{x_j} - \bs x)^{\bs r}, r \leq \zeta$ is a basis of those polynomials the other direction follows.
\end{proof}

\begin{lemma}\label{lem:KLdivergence}
    Let $(\Omega_i, \mc F_i)$, $i=1,2$ be measurable spaces, let $T: \Omega_1 \to \Omega_2$ be an injective, measurable map for which $\Tinv$ defined by $\Tinv(y) = T^{-1}(y)$, $y \in T(\Omega_1)$ and $\Tinv(y) = x_0$, $y \not \in T(\Omega_1)$ for some fixed $x_0 \in \Omega_1$ is also measurable. Let $P,Q$ be probability measures on $(\Omega_1, \mc F_1)$ for which $P \ll Q$, and denote by $P_T = P \circ T^{-1}$ and $Q_T = Q \circ T^{-1}$ the image measures on $(\Omega_2, \mc F_2)$ under $T$. Then for the Kullback-Leibler divergence, 
    $$\mathrm{KL} (P  \mid Q) = \mathrm{KL} (P_T  \mid Q_T) .$$
\end{lemma}

\begin{proof}
For a $\sigma$-finite measure $\mu$ let $p = \dx P / \dx \mu$ and $q = \dx Q / \dx \mu$ denote the Radon-Nikodým derivatives. 
Let $\tilde p \defeq p \circ \Tinv $. Since $\Tinv \circ T = id_{\Omega_1}$ we have that
\begin{align*}
 P_T(A) & = \int_{\Omega_1} 1_A \circ T\ p \, \dx \mu = \int_{\Omega_1} 1_A \circ T\ \tilde p \circ T \, \dx \mu\\
    & = \int_{\Omega_2} 1_A \, \tilde  p \, \dx \mu_T. 
\end{align*}
Therefore $\tilde p = \dx P_T / \dx \mu_T$ and similarly $\tilde q : = q \circ \Tinv = \dx Q_T / \dx \mu_T$. Then
\begin{align*}
    \mathrm{KL} (P_T  \mid Q_T) & = \int_{\tilde p >0}\, \Big(\log\big(\tilde p(y) / \tilde q(y) \big)\Big)\, \tilde p(y)\, \dx \mu_T(y)\\
    & = \int_{\tilde p\circ T >0}\, \Big(\log\big(\tilde p\circ T (x) / \tilde q\circ T (x) \big) \Big)\, \tilde p\circ T (x)\, \dx \mu(x)\\
    & = \int_{p >0}\, \Big(\log\big( p (x) / q (x) \big) \Big)\, p(x)\, \dx \mu(x) = \mathrm{KL} (P  \mid Q) 
\end{align*}
\end{proof}

\section{Additional numerical illustrations}\label{sec:simfullint}

Here we present the results for the Figure \ref{fig:deriv_bw_comp} on the whole of $[0,1]$. In Figure \ref{fig:deriv_bw_comp_full_interval} the error variance is $0.5$ and in Figure \ref{fig:deriv_bw_comp_full_interval_small_errors} the error variance is $0.1$.

\begin{figure}[h!]
	\begin{subfigure}[b]{0.49\linewidth}
		\includegraphics[width=\linewidth]{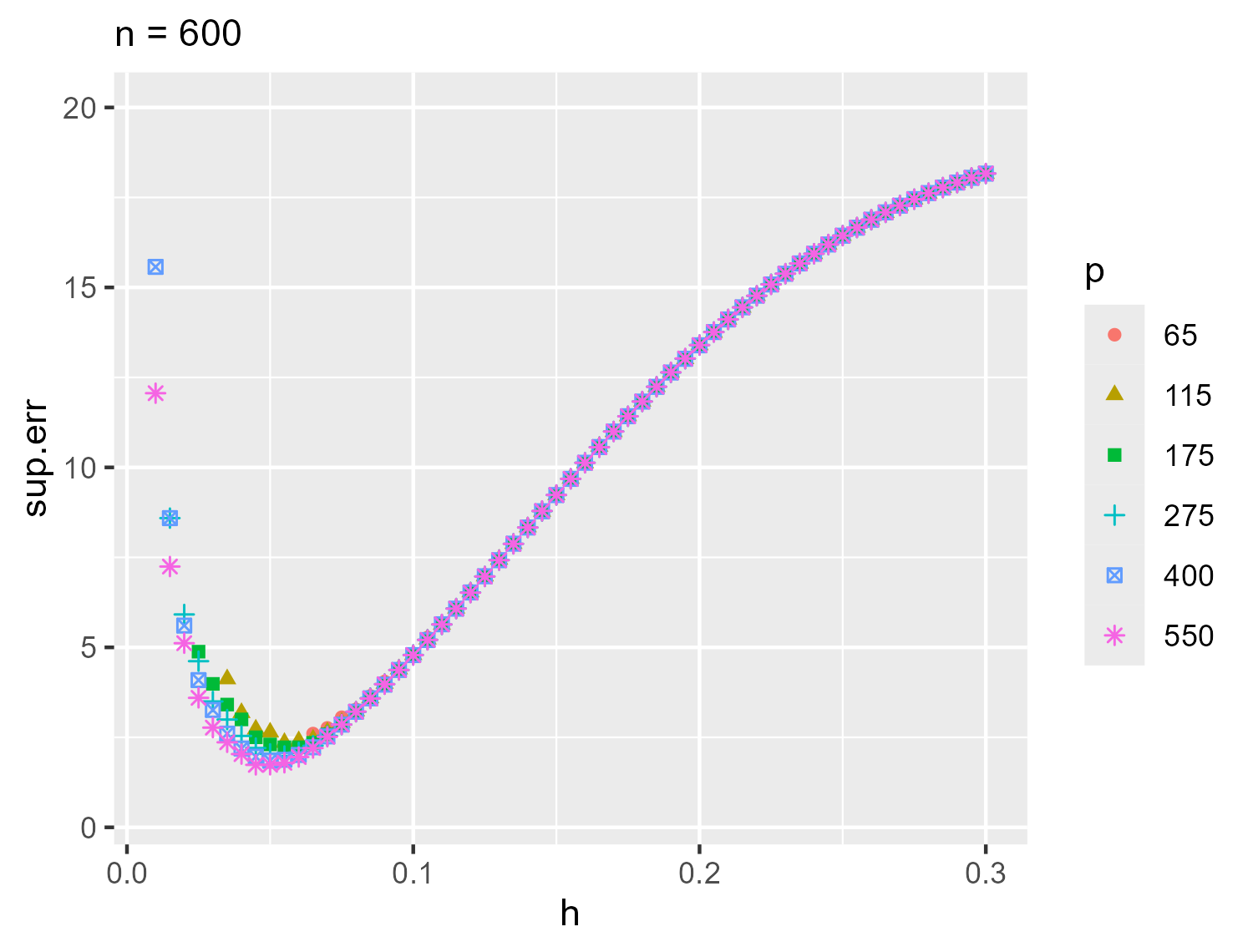}
		\caption{Local quadratic estimator.}
		\label{fig:deriv_bw_comp_quad_full_interval}
	\end{subfigure}
	\hfill
	\begin{subfigure}[b]{0.49\linewidth}
		\includegraphics[width=\linewidth]{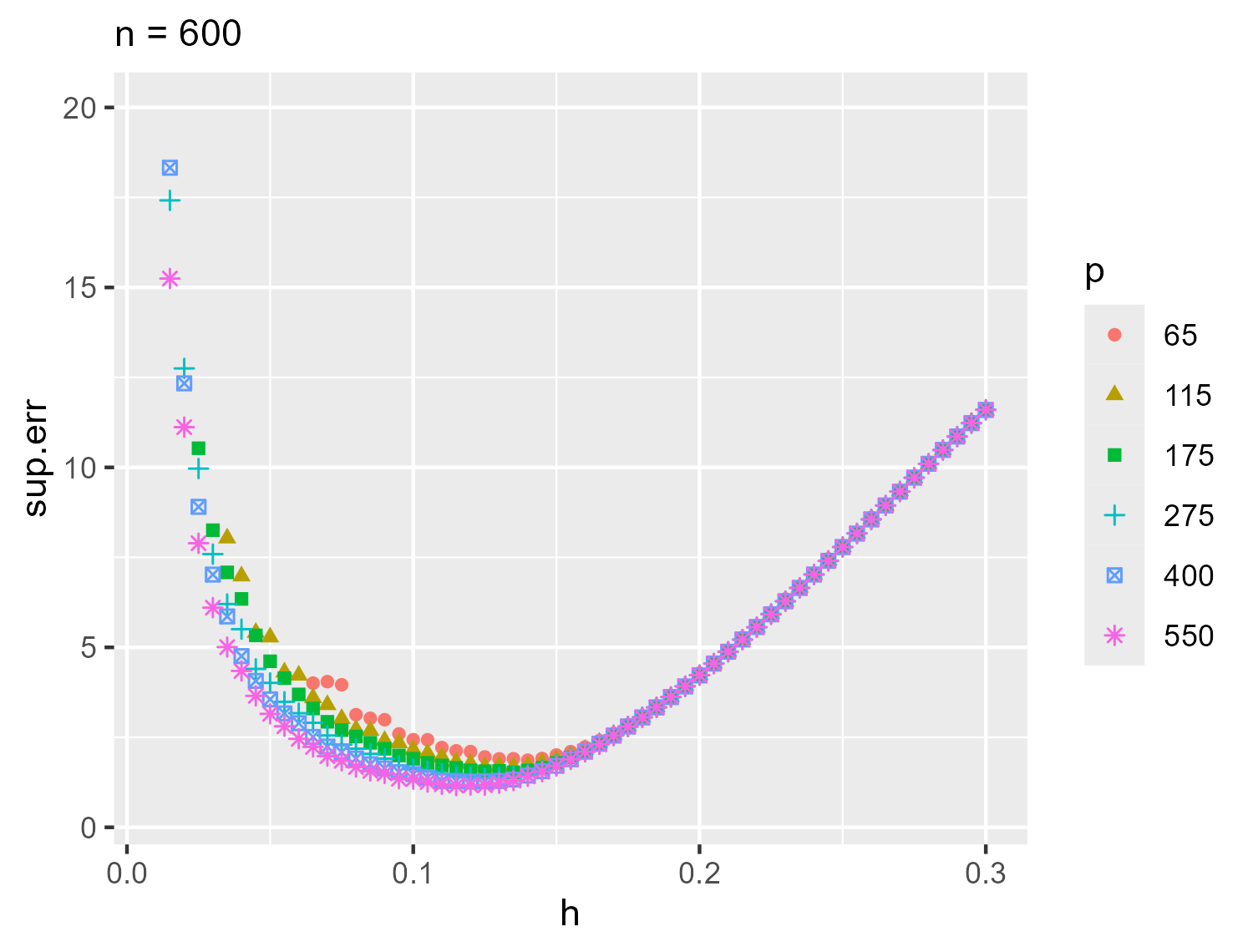}
		\caption{Local cubic estimator.}
		\label{fig:deriv_bw_comp_cubic_full_interval}
	\end{subfigure}%
	\caption{Bandwidth comparison as in Figure \ref{fig:deriv_bw_comp} for the local quadratic (a) and cubic (b) estimator for different $p$ and $n = 600$ with the sup-error taken on the whole of $[0,1]$. Error variance $\sigma = 0.5$.}
	\label{fig:deriv_bw_comp_full_interval}
\end{figure}

\begin{figure}[h!]
	\begin{subfigure}{0.49\linewidth}
		\includegraphics[width=\linewidth]{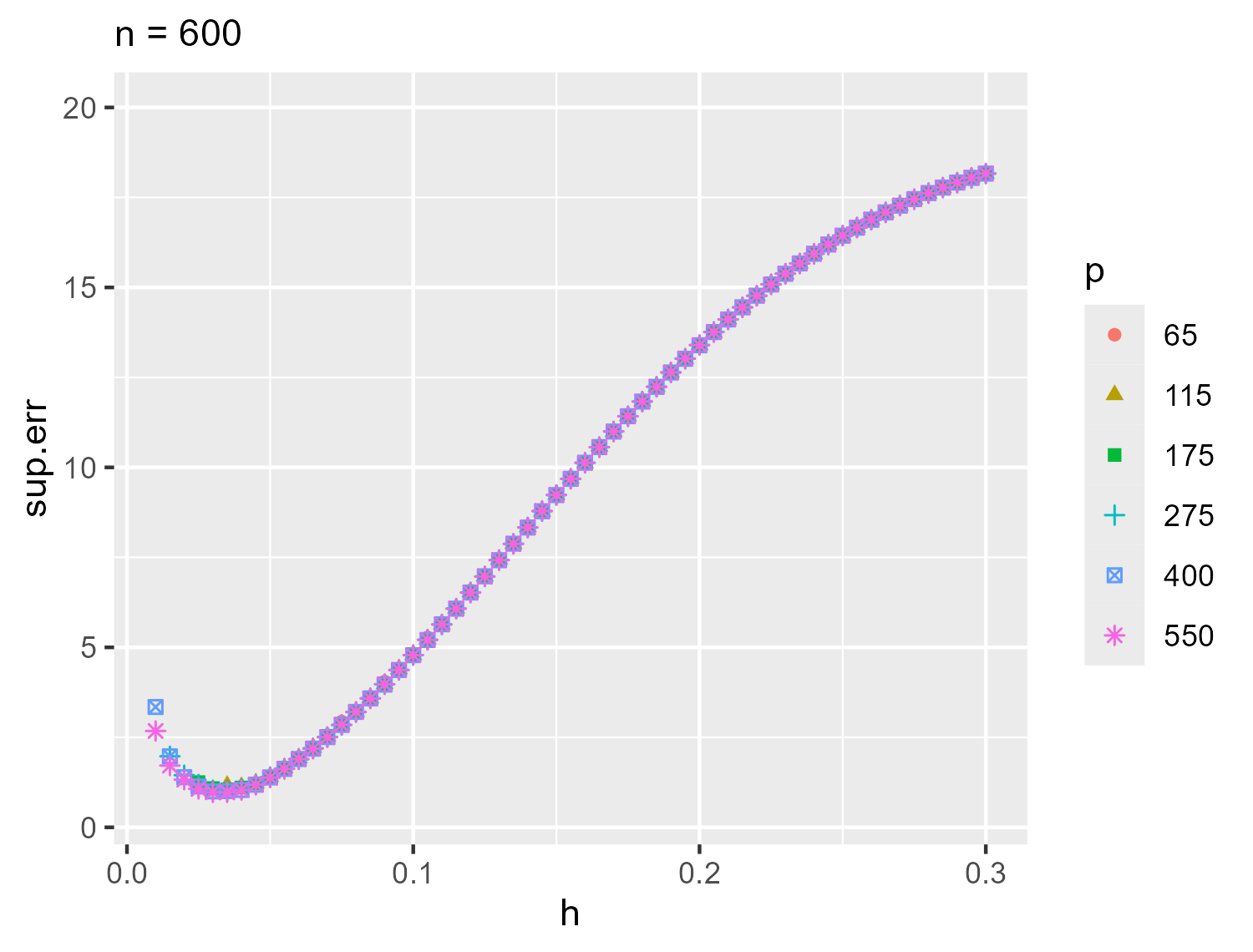}
		\caption{Local quadratic estimator.}
		\label{fig:deriv_bw_comp_quad_full_interval_small_errors}
	\end{subfigure}
	\hfill
	\begin{subfigure}{0.49\linewidth}
		\includegraphics[width=\linewidth]{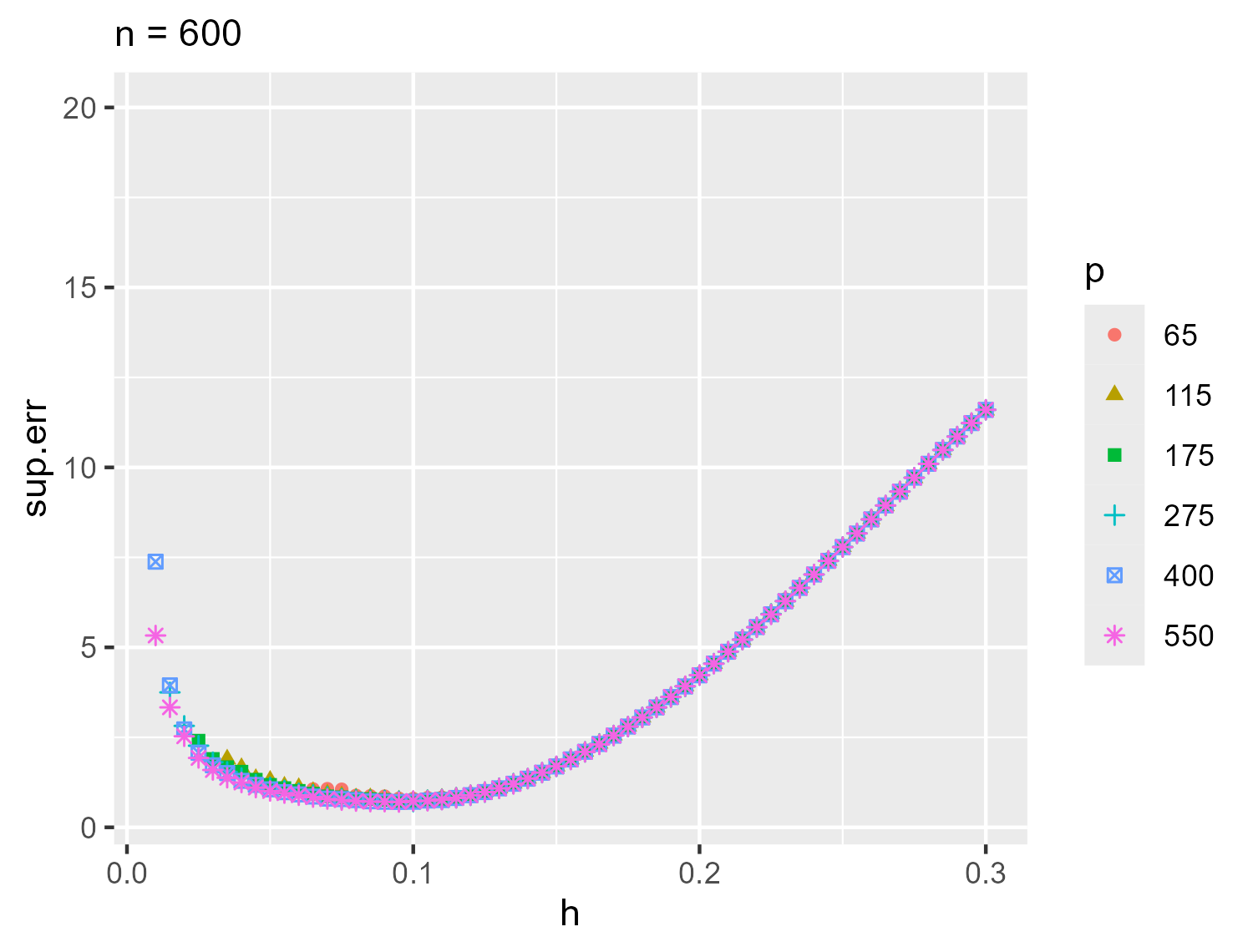}
		\caption{Local cubic estimator.}
		\label{fig:deriv_bw_comp_cubic_full_interval_small_errors}
	\end{subfigure}%
	\caption{Same bandwidth comparison as in Figure \ref{fig:deriv_bw_comp_full_interval} but with a smaller error variance $\sigma = 0.1$.}
	\label{fig:deriv_bw_comp_full_interval_small_errors}
\end{figure}

\begin{figure}
    \centering
    \includegraphics[width=\linewidth]{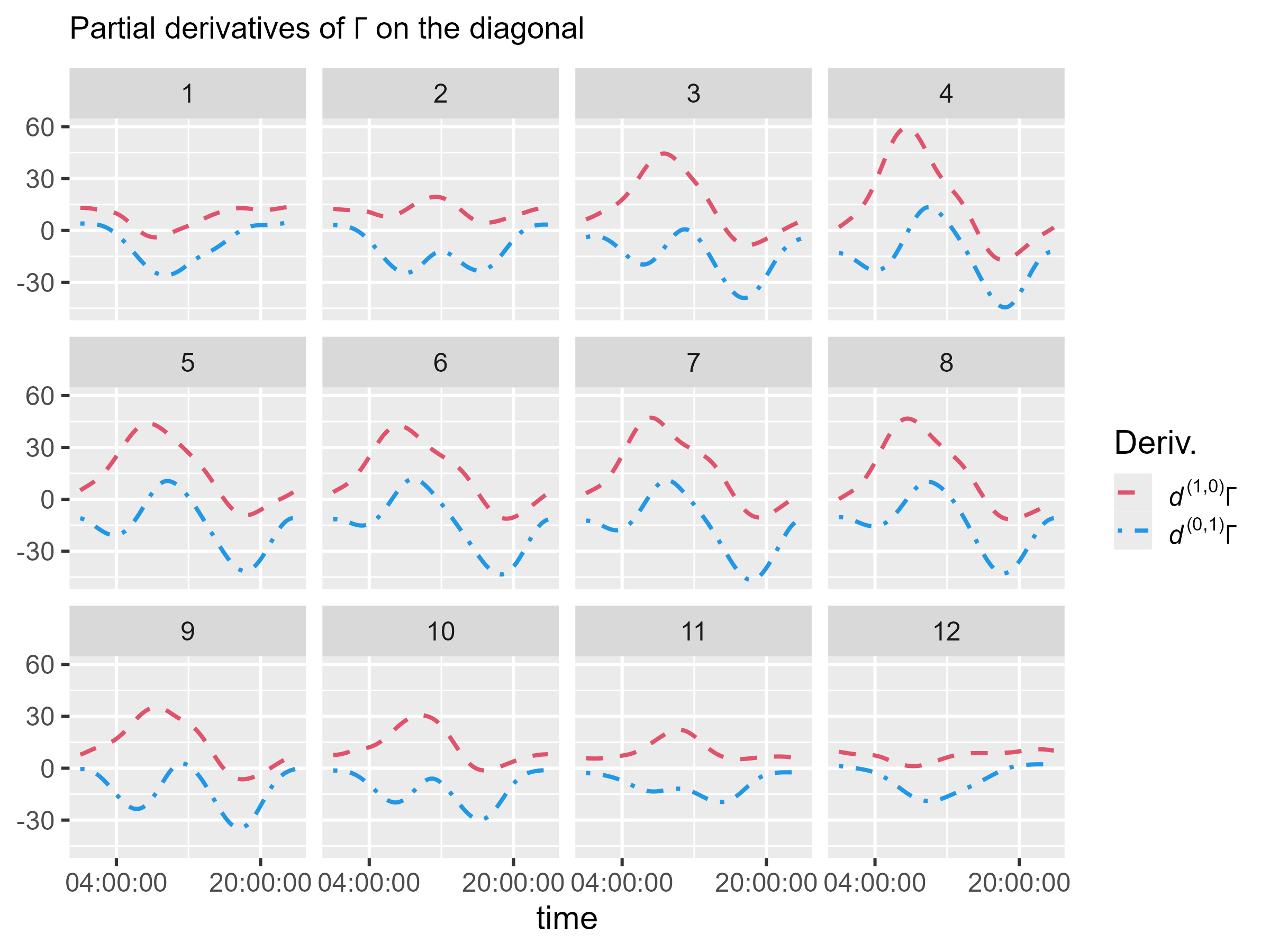}
    \caption{Estimates $ \partial^{(0,1)}_{u}\Gamma(x,x)$ and $\partial^{(1,0)}_{u}\Gamma(x,x)$ of the covariance kernel $\Gamma$ for temperatures in all months.}
    \label{fig:weather_cov_diag_all_months}
\end{figure}

\end{document}